\documentclass[10pt]{article}
\usepackage[utf8]{inputenc}

\usepackage{amsmath,amsfonts,amsthm,amscd,amssymb,thmtools,mathabx}
\usepackage{bm}
\usepackage{cases}
\usepackage[hypertexnames=false]{hyperref}
\usepackage{cite}
\usepackage{xcolor}

\numberwithin{equation}{section}

\usepackage{cleveref}
\crefname{equation}{}{}

\usepackage{enumitem}

\usepackage{graphicx}

\usepackage{authblk}

\declaretheorem[
name=Theorem,
numberwithin=section
]{theorem}

\declaretheorem[
name=Lemma,
sibling=theorem
]{lemma}

\declaretheorem[
name=Proposition,
sibling=theorem
]{proposition}

\declaretheorem[
name=Corollary,
sibling=theorem
]{corollary}

\theoremstyle{definition}

\declaretheoremstyle[
headfont=\normalfont\bfseries,
bodyfont=\normalfont,
spaceabove=6pt,
spacebelow=6pt,
headpunct={.},
postheadspace=0.5em,
qed=\(\diamond\)
]{mystyle-remark}

\declaretheorem[
style=mystyle-remark,
name=Remark,
sibling=theorem
]{remark}

\newcommand{\RR}{\mathbb{R}}

\def\ii{\mathrm{i}}
\def\dd{\mathrm{d}}
\def\ee{\mathrm{e}}

\DeclareMathOperator{\Tr}{Tr}

\begin{document}

\title{Modified Scattering for the Time-Dependent Kohn--Sham Equation}

\author[1]{Masaki Kawamoto}
\author[2]{Jinyeop Lee}
\author[3]{Changhun Yang}
\author[4]{Chanjin You}

\date{}

\renewcommand{\Affilfont}{\small}

\affil[1]{Research Institute for Interdisciplinary Science, Okayama University, 3-1-1, Tsushimanaka, Kita-ku, Okayama City, Okayama, 700-8530, Japan\\
    \href{mailto:kawamoto.masaki@okayama-u.ac.jp}{kawamoto.masaki@okayama-u.ac.jp}}

\affil[2]{Department of Applied Mathematics, Kyung Hee University, 1732 Deogyeong-daero, Giheung-gu, Yongin-si, Gyeonggi-do, South Korea\\
	\href{mailto:jinyeop.lee@khu.ac.kr}{jinyeop.lee@khu.ac.kr}}

\affil[3]{Department of Mathematics, Chungbuk National University, Cheongju 28644, South Korea\\
	\href{mailto:chyang@chungbuk.ac.kr}{chyang@chungbuk.ac.kr}}

\affil[4]{Department of Mathematics,
Yale University, 219 Prospect Street, New Haven, CT 06511, USA\\
\href{mailto:chanjin.you@yale.edu}{chanjin.you@yale.edu}}

\renewcommand\Authands{, and }

\maketitle

{
	\let\thefootnote\relax
	\footnotetext{All authors contributed equally; authors are listed in alphabetical order.}
}

\begin{abstract}
	We study the long-time behavior of the (critical) Kohn--Sham equation in two and three dimensions, i.e.,
	\[
	\ii \partial_t {\gamma} = \Big[-\frac{1}{2}\Delta + \lambda \, |\cdot|^{-1} \ast \rho_{{\gamma}} +  \mu \, \rho_{{\gamma}}^{1/d}, {\gamma} \Big] \quad \text{for} \quad d=2,3.
	\]
	By introducing a suitable ``square root'' of the density matrix and exploiting the pseudo-conformal transform, we establish global well-posedness for small initial data in an appropriate weighted Schatten norm.
	We also prove the optimal time decay of the particle density and establish modified scattering for small and localized solutions.
	In particular, our results provide a resolution to the open problems proposed by Pusateri and Sigal (2021) for the critical and subcritical regimes, rigorously proving their conjectures regarding modified scattering in the critical case and linear scattering in the subcritical cases.
	Our results place these scattering phenomena in the operator-valued setting of density matrices, thereby extending the classical scalar theory to a broader framework.
\end{abstract}

\section{Introduction}

\subsection{Background and Previous Results}

Consider the spatial dimension $d \in \{2, 3\}$. We study a time-dependent density matrix $\gamma(t)$ acting on the Hilbert space $L^2(\mathbb{R}^d)$, which we assume to be a positive semi-definite, trace-class operator. Let $\gamma(t, x, y)$ denote the associated integral kernel\footnote{Recall that the integral kernel $\gamma(x,y)$ of an operator $\gamma$ on $L^2(\mathbb{R}^d)$ is a measurable function $\mathbb{R}^d \times \mathbb{R}^d \to \mathbb{C}$ such that, for any $f \in L^2(\mathbb{R}^d)$, the action of the operator is given by
	\[
	(\gamma f)(x) = \int_{\mathbb{R}^d} \gamma(x,y) f(y) \, \mathrm{d}y.
	\]} of $\gamma(t)$. Where there is no ambiguity, we will suppress the time dependence and simply write $\gamma$ or $\gamma(x,y)$ for simplicity.

We study the time-dependent Kohn--Sham equation
\begin{equation}
	\label{eqn:KS-gamma}
	\begin{cases}
		\begin{aligned}
			\ii \partial_{t} \gamma
			&=
			\Big[
			-\frac{1}{2}\Delta + V(\rho_{\gamma}),
			\, \gamma
			\Big], \\
			\gamma|_{t=1} &= \gamma_1,
		\end{aligned}
	\end{cases}
\end{equation}
where the density is defined by $\rho_\gamma(t,x) = \gamma(t,x,x)$. The potential consists of two parts
\[
V = V_{\text{H}} + V_{\text{xc}},
\]
with
\[
V_{\text{H}}(\rho) = \lambda \, |\cdot|^{-1} * \rho
\quad\text{and}\quad
V_{\text{xc}}(\rho) = \mu \, \rho^{1/d},
\]
where $\lambda,\mu \in \mathbb{R}$.

The term $V_{\text{H}}$ represents the Coulomb (Hartree) interaction, while $V_{\text{xc}}$ corresponds to a local density approximation of exchange correlation effects. The physical background and heuristic derivation of \Cref{eqn:KS-gamma} are discussed in \Cref{sec:motivation}.
In this paper, we investigate the long-time behavior of solutions to \Cref{eqn:KS-gamma} for sufficiently small initial data.

The nonlinear potential in \eqref{eqn:KS-gamma} can be generalized to
\begin{equation}\label{generalized potential}
	\widetilde{V} = \widetilde{V}_{\text{H}} + \widetilde{V}_{\text{xc}},
\end{equation}
where
\[
\widetilde{V}_{\text{H}}(\rho) = \lambda \, W * \rho
\quad\text{and}\quad
\widetilde{V}_{\text{xc}}(\rho) = \mu \,\rho^{\beta},
\]
for some interaction kernel $W$ on $\mathbb{R}^d$ and
$\beta>0$.
The existence of global solutions to \eqref{eqn:KS-gamma} and their long-time behavior depend strongly on the class of interaction kernels and the range of the exponent $\beta$.
Among various types of interaction kernels, a prototypical example is given by $W(x)=|x|^{-\alpha}$ with $0<\alpha<d$.
Our model corresponds to the case $\alpha=1$ and $\beta=1/d$ which is not only physically relevant but also mathematically distinguished, as it lies at the threshold between short-range and long-range behavior, often referred to as the \textit{scattering-critical} regime. On the other hand, the regime $1<\alpha<d$ and $\beta>1/d$
is typically referred to as the short-range (or \textit{subcritical}) regime, in which solutions are expected to exhibit linear scattering, namely, to behave asymptotically like solutions to the corresponding linear equation. At the scattering-critical level, however, linear scattering generally fails, and one instead expects a modified asymptotic behavior involving a nonlinear phase correction. We also discuss related subcritical scattering results in Appendix~\ref{sec:subcritical}. See Section~\ref{subsubsec:Dispersive
	PDE Approach} below for a more detailed discussion of the scattering behavior.

For mixed states, where the density matrix is trace class and may have either finite or infinite rank, the Cauchy problem for Hartree, Hartree--Fock, and Kohn--Sham equations with generalized interaction potentials has been extensively studied in the literature. In particular, global well-posedness for trace-class initial data was established in  \cite{Bove1974,Bove1976,Chadam1976,Chadam1975,Zagatti1992,Jerome2015,Sprengel2017}. Modified scattering, as well as the construction of modified wave operators, for the finite-rank Hartree--Fock equation with Coulomb interaction $W(x)=|x|^{-1}$ in dimensions $d\geq 2$ was established in \cite{Wada2002,Ikeda2012}. For the Kohn--Sham equation \Cref{eqn:KS-gamma}, small-data scattering in almost all subcritical cases was obtained in \cite{Pusateri2021}. More recently, modified scattering for the three-dimensional infinite-rank Hartree equation with Coulomb interaction was proved in \cite{Nguyen2024}. Asymptotic stability of the Hartree equation in the semiclassical regime was studied in \cite{Hadama2025,Hadama2025a}. For the one-dimensional Hartree--Fock equation with an even finite-measure potential, \cite{Maleze2025} observed a cancellation between the direct and exchange terms for plane waves, leading to scattering instead of modified scattering.

From a dynamical viewpoint, scattering phenomena may be regarded as a stability problem around the vacuum state. Beyond this perturbative regime, one may also study the stability and long-time behavior near nontrivial equilibria. In this direction, the Hartree equation admits important translation-invariant stationary states describing infinite quantum systems, and the dynamics around such steady states has been studied in \cite{Lewin2014,Lewin2015,Chen2017,Chen2018,Lewin2020,Collot2020,Collot2022,Hadama2023,Hadama2024,Hadama2025a,Hadama2025,You2024,Smith2024,Nguyen2025,Hadama2025b,Borie2025}.

\subsection{Main Result}\label{sec:main-result}

We now state our main theorem.
\begin{theorem}\label{thm:main1}
	
	Assume $d=2,3$. Let $m \in (d/2, \, 1+2/d)$ and assume that the initial density matrix $\gamma_1$ satisfies
	\[
	\| \langle \mathrm{J}(1) \rangle^m \gamma_1 \langle \mathrm{J}(1) \rangle^m \|_{\mathfrak{S}^1} + \| \rho_{\gamma_1}\|_{L^{\infty}} \le \epsilon_0.
	\]
	
	For sufficiently small $\epsilon_0>0$, the Cauchy problem \Cref{eqn:KS-gamma} has a unique global-in-time solution $\gamma \in C([1,\infty); \mathfrak{S}^1)$ with $|\mathrm{J}|^m \gamma |\mathrm{J}|^m \in C([1,\infty);\mathfrak{S}^1)$
	\footnote{$|\mathrm J_x (t)|^m= \ee^{\ii t\Delta_x/2}\, |x|^m \,\ee^{-\ii t\Delta_x/2}$ is the Heisenberg evolution of the position operator under the free Schr\"odinger flow. See Section~\ref{subsec:pseudo-conformal}.}
	and its associated density satisfies the time decay estimate
	\[
	\| \rho_{\gamma}(t) \|_{L^p_x} \lesssim \epsilon_0  t^{-d(1-1/p)}
	\]
	for $p \in [1,\infty]$ and $t \ge 1$.
	Moreover, there exist a unique $\gamma_{\infty} \in \mathfrak{S}^1$ and a real-valued function $g_{\infty}$ such that, for some $\delta' >0$,
	\[
	\| \gamma(t)- \ee^{-\ii(-t\Delta/2 + g_{\infty} (-\ii \nabla ) \log t)} \gamma_{\infty} \ee^{\ii (-t\Delta/2 + g_{\infty} (-\ii \nabla ) \log t)} \|_{\mathfrak{S}^1}
	\lesssim \epsilon_0  t ^{-\delta'}
	\]
	holds for all $t \ge 1$.
\end{theorem}

\Cref{thm:main1} establishes global well-posedness and modified scattering for small and localized solutions to \eqref{eqn:KS-gamma}. In particular, this resolves Conjecture 2 in \cite{Pusateri2021} in the physically relevant dimensions $d=2,3$ for scattering-critical Kohn--Sham equations. The subcritical regime corresponding to Conjecture 1 in \cite{Pusateri2021} is also resolved in Appendix~\ref{sec:subcritical}.
See Remark~\ref{remark:Higher dimensions} for a discussion on the higher-dimensional case.

\begin{remark}\phantom{}
	\begin{enumerate}[label=(\arabic*)]
		\item Since $\gamma$ is a nonnegative operator, its square root $\sqrt{\gamma}$ is well-defined. Moreover, one has
		\[
		\| \rho_{\gamma} \|_{L^{\infty}}^{1/2}
		=
		\| \sqrt{\gamma} \|_{L^{\infty}_x L^2_y},
		\]
		since
		\[
		\rho_{\gamma}(x) = \int |\sqrt{\gamma}(x,y)|^2 \, \dd y.
		\]
		Thus, the smallness assumption on the density can equivalently be expressed in terms of $\| \sqrt{\gamma_1} \|_{L^{\infty}_x L^2_y}$.
		\item An alternative formulation of the smallness assumption can be given in terms of weighted trace-class norms, which can be viewed as more natural from a physical perspective.
		More precisely, if the initial data satisfies
		\[
		\| \langle \mathrm{J}(1) \rangle^m \gamma_1 \langle \mathrm{J}(1) \rangle^m \|_{\mathfrak{S}^1} + \| \langle \nabla \rangle^{m} \gamma_1 \langle \nabla \rangle^m \|_{\mathfrak{S}^1} \le \epsilon_0,
		\]
		then the same conclusions as in Theorem~\ref{thm:main1} hold.
		Indeed, one has
		\[
		\| \sqrt{\gamma}(x,y) \|_{L^{\infty}_x L^2_{y}} \lesssim \| \sqrt{\gamma}(x,y) \|_{H^m_{x} L^2_{y}} \sim \| \langle \nabla \rangle^m \gamma \,\langle \nabla \rangle^m \|_{\mathfrak{S}^1}^{1/2}.
		\]
		Moreover, in this case, the propagation of $\| \langle \nabla \rangle^{m} \gamma(t) \langle \nabla \rangle^m \|_{\mathfrak{S}^1}$ is expected to hold but we omit the proof.
		\item The choice of the initial time at $t=1$ is made for notational convenience in the analysis of long-time behavior, in particular to simplify the expressions arising in the computations. The same result can be obtained for initial data posed at $t=t_0 \ (t_0>0)$ with only minor modifications.
		
		For initial data given at $t=0$, the same result can be obtained under the additional regularity assumption
		\[
		\| \langle x \rangle^m \gamma_{t=0} \langle x \rangle^m \|_{\mathfrak{S}^1} + \| \langle \nabla \rangle^{m} \gamma_{t=0} \langle \nabla \rangle^m \|_{\mathfrak{S}^1} \le \epsilon_0,
		\]
		which is the same assumption as in \cite{Pusateri2021}.
		However, the additional regularity assumption can be removed by establishing persistence properties observed by Nahas and Ponce in \cite{NahasPonce2009}, which would require the use of Strichartz estimates.
		
		\item The time decay of the density is optimal in the sense that it coincides with the decay rate of the corresponding linear solution. Moreover, it agrees with the optimal decay rate in the rank-one case.
		\item The function $g_\infty$ arises from the leading-order long-range interaction in the scattering-critical regime. More precisely, the critical nonlinear interaction produces a non-integrable contribution of order $t^{-1}$, whose time integration yields a logarithmic phase correction. The resulting logarithmic phase correction is consistent with the long-range asymptotics previously established in the rank-one case. The function $g_\infty$ is obtained by extracting this leading contribution through an ODE-type asymptotic analysis; see Section~\ref{subsubsec:Dispersive PDE Approach} for a heuristic derivation in the rank-one case.
	\end{enumerate}
\end{remark}

\subsection{Strategy of Proof}
\subsubsection{Half-Density Kernel Formulation}
One of the main ideas is to factorize the density matrix by introducing a ``half-density'' $\kappa \in \mathfrak{S}^2$ such that
$$\gamma = \kappa \kappa^{\ast}.$$
We consider the following evolution equation for $\kappa$:
\begin{equation}
	\label{eqn:KS-kappa}
	\begin{cases}
		\begin{aligned}
			&\ii \partial_{t} \kappa = \left(- \frac{1}{2}\Delta + V(\rho_{\kappa\kappa^{\ast}})\right) \kappa, \\
			&\kappa\vert_{t=1} = \kappa_1.
		\end{aligned}
	\end{cases}
\end{equation}
More precisely, for given initial data $\gamma_1$, we choose $\kappa_1$ satisfying $\gamma_1=\kappa_1\kappa_1^*$. Then, one readily verifies that a solution $\kappa(t)$ to \eqref{eqn:KS-kappa} with such initial data generates, at least formally, a solution to \eqref{eqn:KS-gamma} (see the proof of Theorem~\ref{thm:main1} for rigorous justification).

The main advantage of this formulation is that the evolution no longer involves a commutator structure, so that the Hamiltonian $H_{\gamma} = -\Delta/2 + V(\rho_{\kappa\kappa^{\ast}})$ acts only on the $x$-variable of the kernel $\kappa(t,x,y)$.
In particular, \eqref{eqn:KS-kappa} can be written explicitly in terms of its kernel variables as follows:
\begin{equation}\label{eqn:KS-kappa-kernel}
	\begin{cases}
		\begin{aligned}
			&\ii \partial_t \kappa(t,x,y) = -\frac{1}{2} \Delta_x \kappa(t,x,y) + V(\rho_{\kappa\kappa^{\ast}})(t,x) \kappa(t,x,y), \\
			& \kappa\vert_{t=1}(x,y) = \kappa_1(x,y).
		\end{aligned}
	\end{cases}
\end{equation}
This formulation makes the underlying PDE structure more transparent.
In what follows, we focus on the analysis of this half-density equation \eqref{eqn:KS-kappa}.

\begin{remark}
	Due to gauge freedom, this factorization is inherently non-unique. In particular, one possible choice for $\tilde\kappa$ is an operator governed by the following von Neumann equation
	\begin{equation}\label{von Neumann equation}
		\ii \partial_t \tilde\kappa = \left[H_{\gamma} , \, \tilde\kappa\right],
	\end{equation}
	which was the formulation adopted in \cite{Pusateri2021}. While this approach preserves the commutator structure and may appear more natural, our formulation is more amenable to analysis within the framework of nonlinear PDEs.
\end{remark}

Compared to the von Neumann formulation \eqref{von Neumann equation}, our approach allows us to handle the low-regularity structure of $V_{\mathrm{xc}}$ more directly by estimating the nonlinear term in Sobolev spaces with fractional regularity; see Section~\ref{sec:bounds-nonlinear}.

In our analysis, we take the $L^2$-norm of $\kappa(t,x,y)$ with respect to the $y$-variable. As a consequence, the key quantity we analyze is the density as a function of a single spatial variable \[ \rho_{\kappa\kappa^*}(t,x) = \int_{\mathbb{R}^d} |\kappa(t,x,y)|^2 \, \mathrm{d}y, \] which coincides with $\rho_\gamma$. While the equation is posed at the level of $\kappa$, our analysis primarily focuses on the associated density $\rho_{\kappa\kappa^*}$.

\begin{remark}\label{remark:L2valuedsetting}
	This formulation can be viewed as a Schrödinger-type equation. Indeed, by identifying
	$\kappa(t,x,y)$ with an $L^2$-valued function
	$\mathcal U(t):\mathbb{R}^d \rightarrow L^2(\mathbb{R}^d)$ via  $$\mathcal U(t,x):=\kappa(t,x,\cdot) \in L^2(\mathbb{R}^d),$$
	and rewriting the operator-valued equation \eqref{eqn:KS-kappa} in terms of the $L^2$-valued function $\mathcal{U}$, we arrive at the Schr\"{o}dinger-type equation
	\begin{equation}\label{eqn:KS-u}
		\begin{cases}
			\begin{aligned}
				&\ii \partial_t \, \mathcal U = -\frac{1}{2} \Delta \, \mathcal U + V(|\mathcal U|^2) \, \mathcal U, \\
				&\mathcal U\vert_{t=1} = \mathcal U_1,
			\end{aligned}
		\end{cases}
	\end{equation}
	where
	\[
	\rho_{\kappa \kappa^{\ast}}(t,x) = \int |\kappa(t,x,y)|^2 \, \dd y
	=\|\mathcal U(t,x)\|_{L^2}^2=:|\mathcal U|^2.
	\]
	In the rank-one case, this definition of $|\mathcal{U}|^2$ coincides with the usual modulus squared. This viewpoint allows us to reinterpret the kernel $\kappa$ as an $L^2$-valued function and to transfer analytical techniques from scalar nonlinear Schrödinger equations to the present operator-valued setting.
\end{remark}

\begin{remark}\label{rmk on approach via infite system}
An alternative formulation, considered for instance in \cite{Sabin2014}, is obtained by reducing the operator equation to an infinite coupled system of scalar equations through spectral decomposition. Indeed, since $\gamma_1$ is a nonnegative trace-class operator, one may write
\begin{equation}
    \gamma_1
    =
    \sum_{j=1}^\infty
    a_j
    |u_{1,j}\rangle
    \langle u_{1,j}|,
\end{equation}
where $\{u_{1,j}\}_{j=1}^\infty$ is an orthonormal system in $L^2(\mathbb{R}^d)$ and $\{a_j\}_{j=1}^\infty \subset \mathbb{R}^+\cup\{0\}$ satisfies
\[
    \sum_{j=1}^\infty a_j < \infty.
\]
Then \eqref{eqn:KS-gamma} is formally equivalent to the following infinite coupled system: 
\footnote{A direct componentwise application of the argument in the pure state case is not available for the coupled system. For instance, the low-regularity Sobolev estimate
\[
    \left\| |u|^{2/d} u \right\|_{\dot{H}^m}
    \lesssim
    \| u \|_{L^{\infty}}^{2/d} \|u \|_{\dot{H}^m}, \qquad m< 1+ \frac{2}{d}
\]
in the pure state case uses the fact that $z \mapsto |z|^{2/d} z$ is a $C^{1,2/d}$ map. 
In the coupled system \eqref{eq:coupled-system}, however, the nonlinear interaction has the mixed structure
\[
   \left( \sum_{k=1}^{\infty} a_k |u_k|^2 \right)^{1/d} u_j,
\]
for which the corresponding estimate is not directly available, since derivatives may fall on the density, which involves all the components.}
\begin{equation}\label{eq:coupled-system}
   \ii \partial_t u_j=
   -\frac12\Delta u_j
   +
   V\Bigl(\sum_{k=1}^\infty a_k |u_k|^2\Bigr)u_j, \qquad j=1,2,\cdots.
\end{equation}
Introducing the $\ell^2$-valued function $\mathbf{u}(t):\mathbb{R}^d\rightarrow \ell^2(\mathbb{N})$ defined by 
$$[\mathbf{u}(t,x)]_j=\sqrt{a_j}\,u_j(t,x), \quad   j=1,2,\cdots,$$
the coupled system \eqref{eq:coupled-system} can be rewritten as
\begin{equation}\label{eq:vectorvaluedNLS}
   \ii \partial_t \mathbf{u} = -\frac12 \Delta \mathbf{u} + V(\|\mathbf{u}\|_{\ell^2}^2)\mathbf{u} 
\end{equation}
where 
$$ \|\mathbf{u}\|_{\ell^2}^2 := \|\mathbf{u}(t,x)\|_{\ell^2}^2=\sum_{j=1}^\infty a_j|u_j(t,x)|^2.$$
Since the orthogonality of the system is preserved by the evolution, the corresponding solution to \eqref{eqn:KS-gamma} is represented by
\[
    \gamma(t)
    =
    \sum_{j=1}^\infty
    a_j |u_j(t)\rangle\langle u_j(t)|,
\]
and hence
\[
    \rho_\gamma(t,x)
    =
    \|\mathbf{u}(t,x)\|_{\ell^2}^2.
\]

The half-density formulation used in the present paper can be identified with the above $\ell^2$-valued formulation, provided that the corresponding function spaces are chosen consistently. Indeed, letting $\{e_j\}_{j=1}^\infty$ be an orthonormal basis of $L^2(\mathbb{R}^d)$, we may write
\[
    \kappa(t,x,y) = \sum_{j=1}^{\infty} \sqrt{a_j} u_j(t,x) \overline{e_j(y)}.
\]
Then 
\[
    \|\kappa(t,x,\cdot)\|_{L^2_y} = \| \mathbf{u} (t,x) \|_{\ell^2}.
\]
More generally, we take the relevant function spaces to be $L^p_x (\ell^2)$ and $\dot{H}^m_x (\ell^2)$, defined by
\[
    \| \mathbf{u} \|_{L^p_x (\ell^2)}^p
    =
        \int_{\RR^d} 
        \left(
        \sum_{j=1}^{\infty} a_j |u_j(x)|^2
        \right)^{p/2} \, \dd x,
       \qquad
    \| \mathbf{u} \|_{\dot{H}^m_x (\ell^2)}^2 
    =
   \sum_{j=1}^{\infty} a_j \| u_j \|_{\dot{H}^m_x}^2,
\]
so that 
\[
    \| \kappa(t) \|_{L^p_x L^2_y}
    =
    \|\mathbf{u}(t) \|_{L^p_x (\ell^2)},
    \qquad
    \|\kappa(t) \|_{\dot{H}^m_x L^2_y}
    =
    \| \mathbf{u}(t) \|_{\dot{H}^m_x (\ell^2)}.
\]
This differs from the $\ell^2(L_x^p)$-based framework used, for instance, in \cite{BFF2026,Zagatti1992}.

We adopt the half-density formulation because, while being naturally equivalent to the above $\ell^2$-valued formulation, it provides a basis-independent description of mixed states and allows the operator equation to be treated directly using a more compact notation, without explicitly passing through a spectral decomposition or carrying infinitely many indices and weighted sequence norms.
\end{remark}

We now state the corresponding version of Theorem~\ref{thm:main1} in terms of $\kappa$.
\begin{theorem}\label{thm:main2}
	Assume $d=2,3$. Let $m \in (d/2, 1+2/d)$ and assume that the initial data $\kappa_1$ satisfies
	\[
	\| \langle \mathrm{J}_x(1) \rangle^m  \kappa_1 \|_{L^2_{x,y}} + \|  \kappa_1 \|_{L^\infty_x L^2_y} \le \epsilon_0.
	\]
	For sufficiently small $\epsilon_0>0$, the Cauchy problem \Cref{eqn:KS-kappa} has a unique global-in-time solution $\kappa \in C([1,\infty);\mathfrak{S}^2)$ with $|\mathrm{J}_x|^m \kappa \in C([1,\infty);\mathfrak{S}^2)$ satisfying
	the time decay estimate
	\begin{equation}\label{est:timedecay}
		\| \kappa(t) \|_{L^p_x L^2_y} \lesssim \epsilon_0  t^{-d(1/2-1/p)}
	\end{equation}
	for $p \in [2,\infty]$ and $t \ge 1$.
	Moreover, there exist $\kappa_{\infty} \in \mathfrak{S}^2$
	and a real-valued function $g_{\infty}$ such that, for some $\delta' >0$,
	\[
	\| \kappa(t) - \ee^{-\ii(-t\Delta/2 + g_{\infty} (-\ii \nabla ) \log t)} \kappa_{\infty}  \|_{\mathfrak{S}^2}
	\lesssim \epsilon_0 t^{-\delta'}
	\]
	holds for all $t \ge 1$.
\end{theorem}

\begin{remark}\phantom{ }
	\begin{enumerate}[label=(\arabic*)]
		\item The propagator with the phase modification acts only on the $x$-variable.
		\item Functions in $L^2_{x,y}$ can be identified with Hilbert--Schmidt operators.
		\item The time decay estimates in  \Cref{est:timedecay} for $\kappa(t)$ are optimal in the sense that they coincide with the decay rate of solutions to the free Schr\"{o}dinger equation.
	\end{enumerate}
\end{remark}

The proof of \Cref{thm:main1} follows from \Cref{thm:main2} through the relation $\gamma = \kappa \kappa^{\ast}$. The details are provided in \Cref{sec:proof-main1}.
The rest of the paper is devoted to the proof of \Cref{thm:main2}.

\subsubsection{Dispersive PDE Approach}\label{subsubsec:Dispersive PDE Approach}

Once the half-density formulation \eqref{eqn:KS-kappa} is derived, the equation exhibits a structure similar to that of nonlinear Schrödinger equations, as pointed out in Remark~\ref{remark:L2valuedsetting}. In particular, the dispersive dynamics acts only on one spatial variable, while the remaining variable plays the role of a parameter. This allows us to use a dispersive PDE structure in the original operator-valued equation.

To explain the dispersive mechanism behind our analysis, let us first consider the pure-state (rank-one) setting
\[
\gamma(t)=|u(t)\rangle\langle u(t)|
\]
for some function $u:\mathbb{R}^d\to\mathbb{C}$. In this case, the equation \Cref{eqn:KS-gamma} with the generalized potential \eqref{generalized potential} reduces to the nonlinear Schrödinger equation
\begin{align} \label{NLS}
	\ii\partial_t u
	= -\frac12 \Delta u
	+ \lambda \bigl(|x|^{-\alpha} * |u|^2\bigr)u
	+ \mu |u|^{2\beta}u.
\end{align}

The global existence and asymptotic behavior of solutions to \eqref{NLS} have been extensively studied in the dispersive PDE literature. Here, we focus on the scattering behavior of sufficiently small solutions. By scattering in the Sobolev space $H^s(\mathbb{R}^d)$, we mean that there exists an asymptotic state $u_+\in H^s(\mathbb{R}^d)$ such that
\[
\|
\ee^{-\ii t\Delta/2}u(t)-u_+
\|_{H^s(\mathbb{R}^d)}
\to0
\qquad
\text{as }
t\to\infty.
\]

To analyze the asymptotic dynamics, we introduce the profile
\[
f(t):=\ee^{-\ii t\Delta/2}u(t).
\]
Using the Fourier representation together with the stationary phase method, one obtains the following asymptotic expansion as $t\to\infty$:
\begin{align*}
	u(t,x)
	&=
	\frac{1}{(2\pi)^d}
	\int_{\mathbb{R}^d}
	\ee^{\ii x\cdot\xi}
	\ee^{-\ii t|\xi|^2/2}
	\widehat{f}(t,\xi)
	\,\dd\xi
	\\
	&=
	(2\pi \ii t)^{-d/2}
	\ee^{\ii |x|^2/(2t)}
	\widehat{f}\left(t,\frac{x}{t}\right)
	+
	\text{lower-order terms}.
\end{align*}
Substituting this asymptotic expansion into the nonlinear equation and applying the same asymptotic analysis to the nonlinear terms, one formally obtains the schematic leading-order profile equation
\begin{align*}
	\ii \partial_t \widehat{f}(t,\xi)
	&=
	\frac{1}{(2\pi)^d}
	t^{-\alpha}
	\widetilde{V}_{\mathrm{H}}(|\widehat{f}(t,\xi)|^2)
	\widehat{f}(t,\xi)
	\\
	&\quad
	+
	\frac{1}{(2\pi)^{d\beta}}
	t^{-d\beta}
	\widetilde{V}_{\mathrm{xc}}(|\widehat{f}(t,\xi)|^2)
	\widehat{f}(t,\xi)
	+
	\text{lower-order terms}.
\end{align*}
Consequently, the leading-order contribution is expected to behave like
\begin{align*}
	\widehat{f}(t,\xi)
	=
	\widehat{f}(1,\xi)
	&-\ii
	\int_1^t
	\Big(\frac{1}{(2\pi)^d}
	s^{-\alpha}
	\widetilde{V}_{\mathrm{H}}(|\widehat{f}(s,\xi)|^2)
	+
	\frac{1}{(2\pi)^{d\beta}}s^{-d\beta}
	\widetilde{V}_{\mathrm{xc}}(|\widehat{f}(s,\xi)|^2)
	\Big)
	\widehat{f}(s,\xi)
	\,\dd s \\
	&\;\;+
	\text{ lower-order terms.}
\end{align*}

This heuristic computation suggests that the asymptotic behavior is determined by the integrability of the factors $s^{-\alpha}$ and $s^{-d\beta}$. In particular, one expects scattering whenever
\[
\alpha>1
\quad\text{and}\quad
d\beta>1,
\]
while the borderline case
\[
\alpha=1
\quad\text{and}\quad
d\beta=1,
\]
corresponds to the scattering-critical regime considered in the present paper. In this case, the leading-order integral diverges logarithmically, and one expects a nonlinear logarithmic phase correction. More precisely, the modified profile is expected to converge after removing the long-range phase contribution
\[
\Big\| \ee^{-\ii
	\frac{1}{(2\pi)^d}\int_1^t
	\big(
	V_{\mathrm{H}}(|\widehat{u}(s,-\ii\nabla)|^2)
	+
	V_{\mathrm{xc}}(|\widehat{u}(s,-\ii\nabla)|^2)
	\big)
	\frac{\dd s}{s}}\ee^{-\ii t\Delta/2}u(t) - u_+\Big\|_{H^s(\mathbb{R}^d)}  \to 0
\quad
\text{as }
t\to\infty.
\]

These thresholds are indeed known to be sharp in the scalar setting. For the power-type nonlinear Schr\"odinger equation $(\lambda=0)$, scattering to free solutions is known to occur when $\beta>1/d$ under suitable assumptions on the initial data, while scattering fails for $\beta\le1/d$; see for instance \cite{Barab1991,Tsutsumi1984}. In the critical case \(\beta=1/d\), modified scattering was established in the pioneering works \cite{Ozawa1991,Ginibre1994}, and later refined by Hayashi--Naumkin \cite{Hayashi1998a}.

Similarly, for the Hartree equation (\(\mu=0\)), modified scattering occurs at the threshold \(\alpha=1\); see \cite{Hayashi1998a,Hayashi1998}. Moreover, in the Hartree setting, modified scattering has also been established below the critical threshold, namely for \(1/2<\alpha<1\); see \cite{HayashiNaumkin1998,HayashiNaumkin2001,Nakanishi2002a,Nakanishi2002b,Ginibre2014,Ginibre2015}. We also refer to \cite{Kato1987,Cazenave1990} for the well-posedness theory and to \cite{Kato2011,Ifrim2015,Hoose2025,Murphy2021} for further developments on long-range scattering and modified scattering for nonlinear Schr\"odinger equations.

Several approaches have been developed to rigorously justify such modified scattering behavior, including the factorization method of Hayashi--Naumkin \cite{Hayashi1998a}, the space-time resonance method of Kato--Pusateri \cite{Pusateri2013}, and the wave-packet method of Ifrim--Tataru \cite{Ifrim2015}; see for instance the survey article \cite{Murphy2021}. The common feature of these approaches is the combination of weighted energy estimates and an ODE-type asymptotic analysis, which together identify the leading long-range contribution and incorporate it into a nonlinear phase correction.

The main difficulty in the present problem is to adapt such dispersive PDE techniques to the operator-valued setting. One possible approach, considered for instance in \cite{Sabin2014},
is to reduce the operator equation to an infinite coupled system through spectral decomposition. In contrast, our approach works directly at the operator level through the half-density formulation.

Our strategy is motivated by the work of Pusateri--Sigal \cite{Pusateri2021}, where the half-density formulation was first used to study small-data scattering for operator-valued equations.
In particular, they proved linear scattering for \eqref{eqn:KS-gamma} with generalized potential $\widetilde{V}$ in \eqref{generalized potential} in the short-range regime
\[
\alpha>1
\quad\text{and}\quad
\beta>\frac{1}{\min(d,2)},
\]
for $d=2,3$. Their argument directly treated the operator equation without reducing it to an infinite coupled system, but only partial subcritical regimes were covered. In particular, the remaining subcritical regimes when $d=3$ and the critical regime were left as open problems. One of the main limitations of their approach is that the weighted energy argument relies on integer weights.

The fourth author of the present paper previously provided a partial answer to the conjectures in \cite{Pusateri2021} by proving modified scattering for \eqref{eqn:KS-gamma} in the three-dimensional Hartree case $(\mu=0)$ \cite{You2024}.
That work also treated the operator equation directly, but used a different approach based on the space-time resonance method combined with the Hayashi--Naumkin argument, without relying on the half-density formulation. However, the space-time resonance argument is not well suited for power nonlinearities with non-integer exponents.

From a physical viewpoint, however, one naturally expects both the Hartree-type and power-type nonlinearities to appear simultaneously.
The main novelty of the present paper is to combine the half-density formulation with dispersive PDE techniques in a unified framework capable of treating both classes of nonlinearities simultaneously. This framework enables us to resolve the remaining conjectures proposed in \cite{Pusateri2021}, including both the critical modified scattering regime and the remaining subcritical regimes; see Appendix~\ref{sec:subcritical} for the latter.
Our formulation removes the commutator structure and reduces the problem to a dispersive equation acting on a single spatial variable, allowing us to implement scalar dispersive PDE arguments at the operator level.

Another key ingredient of our approach is the use of the pseudo-conformal transformation.
In the classical Hayashi--Naumkin framework, one typically assumes both weighted norms and Sobolev regularity, for instance
\[
\|
\langle x\rangle^m u_1
\|_{L^2}
+
\|u_1\|_{H^m}
<
\infty,
\qquad
m>\frac{d}{2}.
\]
In our pseudo-conformal framework, the Sobolev regularity assumption can be weakened to suitable $L^\infty$-control\footnote{Such a relaxation of Sobolev regularity assumptions plays an important role in the construction of scattering operators; see, for instance, \cite{Carles2001, Carles2025, Hayashi2006, Kawamoto2025, Nakanishi2002a}.}, while the weighted assumption is still required.   After applying the pseudo-conformal transformation, the weighted condition is converted into a Sobolev regularity condition for the transformed unknown. Schematically,
\[
\|
\langle \mathrm{J}_x (t)\rangle^m u
\|_{L^2}
\quad
\longleftrightarrow
\quad
\|
\langle\nabla\rangle^m v
\|_{L^2},
\]
for arbitrary real $m\ge0$, where $v$ is the transformed function corresponding to $u$; see \Cref{subsec:pseudo-conformal} for the corresponding operator-valued formulation. 
Thus, in the transformed equation, the weighted estimate is replaced by a fractional Sobolev estimate, while the remaining pointwise control is encoded through an $L^\infty$-bound. As a result, the analysis can be carried out essentially within the two spaces $H^m$ and $L^\infty$.

This viewpoint is particularly advantageous in the operator-valued setting.
Indeed, if one attempts to work directly with weighted estimates, additional difficulties related to commutator structures may arise.
By contrast, after the pseudo-conformal transformation, the main task reduces to establishing suitable Sobolev-type nonlinear estimates of the form
\begin{equation}\label{energy estimates}
	\|
	|v|^{2\beta}v
	\|_{H^m(\mathbb{R}^d)}
	\lesssim
	\|v\|_{L^\infty(\mathbb{R}^d)}^{2\beta}
	\|v\|_{H^m(\mathbb{R}^d)},
\end{equation}
for the power-type nonlinearity. Such estimates may be viewed as a fractional chain rule type estimate; see, for instance, \cite[Lemma~2.3]{Hayashi1998a} for a precise statement under suitable assumptions on the parameters. The corresponding estimate for the half-density equation is established in Lemma~\ref{lem:powertype}.\footnote{
	A key difficulty is that, unlike the scalar nonlinearity $|v|^{2\beta}v$, the kernel nonlinearity
	\[
	\|\kappa(t,x,\cdot)\|_{L^2_y}^{2\beta}\kappa(t,x,y)
	\]
	does not immediately exhibit the same pointwise composition structure. The main observation is that, after passing to the half-density formulation, the required fractional Sobolev estimates can nevertheless be recovered by treating the problem at the level of Sobolev norms rather than pointwise differentiation.
}
As a consequence, our approach lowers the weighted assumption in \cite{Pusateri2021} from
\[
\|
\langle x\rangle^2 \kappa_1
\|_{L_{x,y}^2}
<
\infty
\]
to
\[
\|
\langle x\rangle^m \kappa_1
\|_{L_{x,y}^2}
<
\infty,
\qquad
m>\frac{d}{2},
\]
which is essential to cover the subcritical scattering regime together with the critical modified scattering regime in dimensions $d=2,3$.

A crucial point is that Sobolev-type nonlinear estimates such as \eqref{energy estimates} remain valid in the operator-valued setting after passing to the half-density formulation. Indeed, since the dispersive dynamics acts only on the $x$-variable, the nonlinear estimates reduce to Sobolev estimates with respect to a single spatial variable. This makes it possible to adapt arguments from scalar nonlinear Schr\"odinger equations and, in particular, the Hayashi--Naumkin argument at the operator level.
By contrast, in the von Neumann-type formulation considered previously, it is unclear whether the corresponding Sobolev estimates remain valid at the critical regularity level, due to additional difficulties related to the commutator structure. This is one of the reasons why we adopt the present half-density formulation instead of the von Neumann formulation.

As a consequence, the combination of the half-density formulation and the pseudo-conformal framework allows us to implement the dispersive PDE strategy developed for scalar nonlinear Schr\"odinger equations in the operator-valued setting. This enables us to resolve the remaining subcritical regimes as well as the critical modified scattering regime.

Furthermore, since the analysis is carried out directly at the operator level, the framework developed here is expected to be useful for future studies of genuinely operator-valued models such as Hartree--Fock and Dirac--Fock equations.

\begin{remark}[Higher dimensions]\label{remark:Higher dimensions}
	The present argument for the power-type nonlinearity is closely tied to the low regularity structure of the nonlinearity and therefore appears to be essentially restricted to dimensions $d=2,3$. Indeed, the analysis relies on Sobolev estimates of chain-rule type together with the embedding
	\[
	H^m(\mathbb{R}^d)\hookrightarrow L^\infty(\mathbb{R}^d),
	\qquad
	m>\frac{d}{2}.
	\]
	On the other hand, for the power-type nonlinearity, the relevant regularity is determined by the map
	\[
	z\mapsto |z|^{2/d}z,
	\]
	whose differentiability is limited to order $1+2/d$. Consequently, the present argument requires
	\[
	\frac d2 < 1+\frac2d,
	\]
	which is satisfied only for $d\le 3$.
	
	This obstruction is specific to the power-type interaction. In the Hartree case $(\mu=0)$, where the nonlinearity has a smoother structure, it is expected that the argument in \cite{You2024} can be extended to all dimensions $d\ge2$ after suitable modifications of the assumptions and estimates. The one-dimensional case is excluded because the Coulomb potential $|x|^{-1}$ is too singular at the origin.
\end{remark}

\subsection{Structure of the Paper}
The rest of the paper is organized as follows.
In \Cref{sec:motivation}, we provide a heuristic derivation of the equation and explain its physical relevance.
In \Cref{sec:prelim}, we list the notations that are frequently used in the paper. Then we introduce the pseudo-conformal transformation, and present the nonlinear iterative scheme.

In \Cref{sec:bootstrap}, we prove \Cref{prop:bootstrap}, which yields the global existence of $\kappa$ on $[1,\infty)$ and the desired time decay estimates.
Finally, in \Cref{sec:modsc}, we establish modified scattering, completing the proof of \Cref{thm:main2} and \Cref{thm:main1}.

\section{Physical Background and Scaling Regimes}\label{sec:motivation}

\subsection{Physical Motivation}

In effective mean-field descriptions of interacting fermionic systems, the potential $V(\rho)$ typically contains two types of contributions. One is a \emph{Hartree term}, which represents the classical mean-field interaction generated by the particle density.
The other is a correction originating from the Pauli exclusion principle, known as the \emph{exchange term}. In this section, we explain heuristically how these two contributions arise.

\subsubsection{Heuristic Derivation of the Hartree Term}

Consider a system of $N$ interacting particles with pair potential $W$. At the many-body level, the Hamiltonian contains the interaction term
\[
\sum_{1\le i<j\le N} W(x_i-x_j).
\]
If the system is described by a one-particle density matrix $\gamma$, the particle density is
\[
\rho_\gamma(x)=\gamma(x,x).
\]
In a mean-field approximation, one assumes that each particle moves in the average potential generated by the spatial distribution of all other particles. The interaction energy then becomes approximately
\[
E_{\mathrm{H}}
\approx
\frac12
\iint
W(x-y)\rho(x)\rho(y)\,\dd x \dd y.
\]
Taking the functional derivative with respect to $\rho$ yields the effective potential
\[
V_{\mathrm{H}}(x)
=
\frac{\delta E_{\mathrm{H}}}{\delta\rho(x)}
=
\int W(x-y)\rho(y)\,\dd y
=
(W * \rho)(x).
\]
Thus the mean-field dynamics naturally contains the Hartree term
\[
V_{\mathrm{H}}(\rho)
=
W * \rho.
\]
In Coulomb systems, i.e., $W(x)=|x|^{-1}$, this becomes
\[
V_{\mathrm{H}}(\rho)
=
|\cdot|^{-1} * \rho,
\]
which corresponds physically to the classical electrostatic potential generated by the charge density. Here $1/|x|$ denotes the three-dimensional Coulomb potential.
For more detailed results, see e.g. \cite{Elgart2004,Benedikter2014,Porta2017} and the references therein.

We adopt this choice in both the three-dimensional setting and the two-dimensional one: in the latter case we consider a 2D sheet embedded in $\mathbb{R}^3$, so the electrons are confined to the sheet but still interact through the full 3D Coulomb kernel $1/|x|$, rather than the logarithmic potential $-\log|x|$ that would arise from a purely two-dimensional Poisson theory.

\subsubsection{Heuristic Derivation of the Exchange Term}

The Hartree approximation neglects the fermionic nature of the particles. Because fermionic many-body states are antisymmetric, the Pauli exclusion principle induces additional correlations even in the absence of direct interactions.

Within Hartree--Fock theory, the interaction energy contains the exchange contribution
\[
E_{\mathrm{xc}}
=
-\frac12
\iint
\frac{|\gamma(x,y)|^2}{|x-y|}
\,\dd x \, \dd y,
\]
where we again use the three-dimensional Coulomb kernel $1/|x-y|$ for both two- and three-dimensional cases, as in the Hartree term.

To estimate its dependence on the density, we consider a locally homogeneous Fermi gas. Near the ground state, momentum states are filled up to the local Fermi momentum $k_F(x)$, so that
\[
\rho(x)
\sim
\int_{|p|\le k_F(x)} \dd p
\sim
k_F(x)^d.
\]
Hence
\[
k_F(x)\sim\rho(x)^{1/d}.
\]
In such a state, the off-diagonal density matrix decays on the \emph{Fermi length}
\[
\ell_F \sim k_F^{-1}.
\]
More precisely,
\[
\gamma(x,y)
\approx
\rho(x)\,
F\!\left(k_F(x)|x-y|\right),
\]
where $F$ is a dimension-dependent function that decays for $|x-y|\gg k_F^{-1}$. Consequently,
\[
|\gamma(x,y)|^2
\sim
\rho(x)^2\,
G\!\left(k_F(x)|x-y|\right),
\]
where $G=|F|^2$.
Substituting this into the exchange energy and writing $r=x-y$ gives
\[
E_{\mathrm{xc}}
\sim
-\iint
\frac{\rho(x)^2\,G(k_F r)}{|r|}
\,\dd r \,\dd x.
\]
The integral over $r$ is effectively restricted to the region $|r|\lesssim \ell_F \sim k_F^{-1}$. Rescaling $u=k_F r$ yields
\[
\int
\frac{G(k_F r)}{|r|}
\,\dd r
\sim
k_F^{1-d}.
\]

Therefore, we have
\[
E_{\mathrm{xc}}
\sim
-\int
\rho(x)^2 k_F^{1-d}\,\dd x.
\]
Using $k_F\sim\rho^{1/d}$, we obtain
\[
E_{\mathrm{xc}}(\rho)
\sim
- \int \rho(x)^{1+1/d} \, \dd x.
\]
Taking the functional derivative with respect to $\rho$ gives the
corresponding effective potential
\[
V_{\mathrm{xc}}(\rho) =
\frac{\delta E_{\mathrm{xc}}}{\delta\rho}
\sim
-\rho^{1/d}.
\]
Thus exchange effects naturally produce nonlinear potentials involving powers of the density determined by the spatial dimension. This contribution is often represented by a local term of the form
\[
V_{\mathrm{xc}}(\rho)=\mu \rho^{1/d},
\]
which captures the leading-order scaling of exchange correlations in the local density approximation, first proposed by Dirac~\cite{Dirac1930}; for rigorous results in this direction, see e.g. \cite{LO1981,FS1990,Bach1992,Solovej2003,LLS2019} and the references therein.

\subsection{Regimes: Small-Data vs.\ Small-Interaction}
In many works in mathematical physics, one normalizes the density matrix by
\[
\Tr \gamma = 1 .
\]
In contrast, dispersive PDE theory often relies on \emph{small-data} assumptions. To connect the two viewpoints one may introduce a scaling
\[
\gamma_\varepsilon := \varepsilon \gamma ,
\]
with $0<\varepsilon\ll1$.
Then
\[
\Tr \gamma_\varepsilon = \varepsilon ,
\]
and any norm that is linear in $\gamma$ is multiplied by $\varepsilon$. This allows one to place the system in a perturbative regime where small-data well-posedness or scattering theory can be applied.

However, the nonlinear potential depends on the density
\[
\rho_\gamma(x)=\gamma(x, x),
\]
so the scaling also affects the nonlinear term.
Then we have
\[
\rho_{\gamma_\varepsilon}(x)=\varepsilon \rho_\gamma(x).
\]

Now, consider our Kohn--Sham equation \Cref{eqn:KS-gamma} for $\gamma_{\varepsilon}$.
Under the scaling $\gamma_\varepsilon=\varepsilon\gamma$ we obtain
\[
\ii \partial_t \gamma_\varepsilon
=
\Big[
-\frac12\Delta + V(\rho_{\gamma_\varepsilon}),\,
\gamma_\varepsilon
\Big],
\]
with
\[
\rho_{\gamma_\varepsilon}=\varepsilon\rho_\gamma .
\]
Consequently,
\[
V_{\mathrm{H}}(\rho_{\gamma_\varepsilon})
=
\lambda |\cdot|^{-1} * (\varepsilon\rho_\gamma)
=
\varepsilon \lambda |\cdot|^{-1} * \rho_\gamma ,
\]
and
\[
V_{\mathrm{xc}}(\rho_{\gamma_\varepsilon})
=
\mu (\varepsilon\rho_\gamma)^{1/d}
=
\mu \varepsilon^{1/d} \rho_\gamma^{1/d}.
\]
Thus the equation becomes
\[
\ii \partial_t \gamma_\varepsilon
=
\Big[
-\frac12\Delta
+ \varepsilon \lambda |\cdot|^{-1} * \rho_\gamma
+ \varepsilon^{1/d}\mu \rho_\gamma^{1/d},
\,\gamma_\varepsilon
\Big].
\]

Therefore, the smallness obtained by the rescaling can be interpreted equivalently as a smallness of the effective coupling constants:
\[
\lambda \mapsto \varepsilon \lambda,
\qquad
\mu \mapsto \varepsilon^{1/d} \mu .
\]
In particular, the ``small-data regime'' in the PDE sense corresponds to a ``weak-coupling regime'' of the interaction.

\section{Preliminaries}\label{sec:prelim}

\subsection{Notations}\label{sec:notations}

We use the following notation throughout the paper.
For $z \in \mathbb{C}$, we denote its complex conjugate by $\bar{z}$.
We write $A\lesssim B$ iff $A \le CB$ for some universal constant $C>0$.
For $x \in \mathbb{R}^d$, we set $\langle x \rangle = (1+|x|^2)^{1/2}$.
We denote the Fourier transform by $\mathcal{F}$, which is defined by
\[
\mathcal{F}f(k) = \int_{\RR^d} \ee^{-\ii k \cdot x} f(x) \, \dd x.
\]
Its inverse is denoted by $\mathcal{F}^{-1}$, which is given by
\[
\mathcal{F}^{-1}g(x) = \frac{1}{(2\pi)^d}\int_{\RR^d} \ee^{\ii k \cdot x} g(k) \, \dd k.
\]
We write $L^2 = L^2(\mathbb{R}^d) = L^2(\mathbb{R}^d;\mathbb{C})$
with the complex inner product
\[
\langle f, g \rangle_{L^2} := \int_{\mathbb{R}^d} f(x) \overline{g(x)} \, \dd x.
\]
We often use the associated real inner product
\[
\langle f, g \rangle:= \Re \langle f, g \rangle_{L^2}.
\]
For $1\le p <\infty$, we denote the $p$-Schatten class on $L^2(\mathbb{R}^d)$ by $\mathfrak{S}^p$ with the norm
\[
\| A \|_{\mathfrak{S}^p} = \big(\operatorname{Tr} (|A|^p) \big)^{1/p}.
\]
When $p=\infty$, we set $\mathfrak{S}^{\infty} = \mathcal{B}(L^2)$, the space of bounded linear operators on $L^2(\mathbb{R}^d)$ with the operator norm
\[
\| A \|_{\mathfrak{S}^{\infty}} = \| A \|_{\mathrm{op}}.
\]

We denote by $L^p_x L^2_y$ the mixed norm space of measurable kernels $\kappa=\kappa(x,y)$, endowed with the norm
\[
\| \kappa \|_{L^p_x L^2_y}
:=\left( \int_{\mathbb{R}^d} \Big( \int_{\mathbb{R}^d} |\kappa(x,y)|^2 \,\dd y \Big)^{p/2} \dd x \right)^{1/p}.
\]
For $m\ge 0$, we define the homogeneous Sobolev norm in the $x$ variable by
\[
\|\kappa\|_{\dot H^m_x L^2_y}
:=
\left(
\int_{\mathbb{R}^d}\int_{\mathbb{R}^d}
\big||\nabla_x|^m \kappa(x,y)\big|^2 \,\dd y\,\dd x
\right)^{1/2},
\]
where
\[
|\nabla_x|^m \kappa(x,y) = \mathcal{F}^{-1} |\xi|^m (\mathcal{F} \kappa)(\xi,y).
\]
For $0<s<2$, we define the homogeneous Besov seminorm in the $x$ variable, with values in $L^2_y$, by
\[
\|\kappa\|_{\dot B^s_{p,q}(L^2)}
:=
\left(
\int_0^\infty
\sup_{|h|\le \tau}
\left(
\int_{\mathbb{R}^d}
\|\Delta_h^2 \kappa(x,\cdot) \|_{L^2_y}^p
\dd x
\right)^{q/p}
\tau^{-1-sq}\,\dd \tau
\right)^{1/q}.
\]
Note that, for $0<s<1$, this seminorm is equivalent to
\[
\|\kappa\|_{\dot B^s_{p,q}(L^2)}
\sim
\left(
\int_0^\infty
\sup_{|h|\le \tau}
\left(
\int_{\mathbb{R}^d}
\|\Delta_h^1 \kappa(x,\cdot) \|_{L^2_y}^p
\dd x
\right)^{q/p}
\tau^{-1-sq}\,\dd \tau
\right)^{1/q},
\]
see \cite[Section~2.5.12]{Triebel1983}.\footnote{Here and throughout the paper, the Besov-type norms are understood with respect to the $x$-variable after taking the $L^2$-norm in the $y$-variable. For simplicity of notation, we omit the explicit distinction between the $x$ and $y$ variables when no confusion can arise.}

Here the difference operators act on the $x$ variable:
\begin{align*}
	\Delta_h^1 \kappa(x,y)&:=\kappa(x+h,y)-\kappa(x,y),\\
	\Delta_h^2 \kappa(x,y)&:=\kappa(x+h,y)+ \kappa(x-h,y)-2\kappa(x,y).
\end{align*}
Finally, we introduce the operators  for $m\ge0$,
\[
|\mathrm J_x (t)|^m:= \ee^{\ii t\Delta_x/2}\, |x|^m \,\ee^{-\ii t\Delta_x/2}
\quad \text{and} \quad \langle\mathrm J_x (t)\rangle^m:= \ee^{\ii t\Delta_x/2}\, \langle x\rangle^m \,\ee^{-\ii t\Delta_x/2},
\]
which correspond to the Heisenberg evolution of the position operator under the free Schr\"odinger flow.

\subsection{Pseudo-Conformal Transformation}\label{subsec:pseudo-conformal}
We introduce a pseudo-conformal transformation $\mathcal{I}: \kappa \mapsto \nu$ associated with the free Schr\"{o}dinger propagator. For a solution $\kappa(t)$ to \eqref{eqn:KS-kappa-kernel} on $[1,T]$, let
\[
\mathrm{U}(t) =\ee^{\ii t \Delta_x/2} , \qquad (\mathrm{M}(t) f)(x,y) = \ee^{\ii |x|^2/(2t)} f(x,y), \qquad (\mathrm{D}(t) f)(x,y) = (\ii t)^{-d/2} f(x/t, y).
\]
Define $\omega(t)$ via $\kappa(t) =\mathrm{M}(t) \mathrm{D}(t) \omega(t) $, which is equivalent to $\omega(t) = \mathcal{F} \mathrm{M}(t) \mathrm{U}(-t) \kappa(t)$. Then there holds
\[
\omega(t,x,y) = (\ii t)^{d/2} \ee^{-\ii t|x|^2 /2} \kappa(t,tx,y).
\]
It follows that
\begin{equation}\label{eqn:pseudo-conformal-prop}
	\| |\nabla_x|^{m} \omega(t)\|_{L^2_{x,y}} =\| |\mathrm{J}_x|^{m} \kappa(t) \|_{L^2_{x,y}}, \qquad
	\| \omega(t) \|_{L^{\infty}_x L^2_y}  = t^{d/2}\| \kappa(t) \|_{L^{\infty}_x L^2_y}
\end{equation}
for $m\ge0$.
A direct computation shows that $\omega$ solves
\begin{equation*}
	\begin{cases}
		\begin{aligned}
			&\ii \partial_t \omega = \left( - \frac{1}{2t^2} \Delta_x + \frac{1}{t} V(\rho_{\omega\omega^{\ast}})\right)\omega, \\[5pt]
			& \omega\vert_{t=1} = \omega_1 = \ii^{d/2} \ee^{-\ii |x|^2/2} \kappa_1.
		\end{aligned}
	\end{cases}
\end{equation*}
We now perform the time inversion. Define $\nu(t) = \overline{\omega(1/t)}$ for $t \in [T^{-1},1]$.
Then $\nu$ solves
\begin{equation}\label{eqn:nu}
	\begin{cases}
		\begin{aligned}
			&\ii \partial_t \nu = \left( - \frac{1}{2} \Delta_x + \frac{1}{t} V(\rho_{\nu\nu^{\ast}})\right) \nu, \\[5pt]
			& \nu\vert_{t=1} = \nu_1 = (-\ii)^{d/2} \ee^{\ii |x|^2/2} \overline{\kappa_1}.
		\end{aligned}
	\end{cases}
\end{equation}
The long-range structure of the equation is reflected in the appearance of the non-integrable coefficient $t^{-1}$ in front of the nonlinear potential.

We use the following local existence theorem, proved in \Cref{sec:LWP} by the contraction mapping principle.
\begin{lemma}\label{lem:LWP}
	Let $d=2,3$ and fix $m \in (d/2,1+2/d)$. Suppose $\nu_1 \in H^m_x L^2_y$.
	Then there exist $T>1$ and a unique solution of \Cref{eqn:nu} such that $\nu \in C([T^{-1}, 1]; H^m_x L^2_y)$.
\end{lemma}

Next, we define a bootstrap norm $X_T$ for $\nu$. For $T \ge 1$, we set
\begin{equation}\label{def:bootstrapnorm}
	\| \nu \|_{X_T}
	:= \sup_{t \in [T^{-1},1]}
	\left\{
	\| \nu(t) \|_{L^2_xL^2_y}
	+  \|\nu(t)\|_{L^{\infty}_x L^2_y}
	+ t^{\delta} \| \nu(t)  \|_{H^m_x L^2_y}
	\right\}.
\end{equation}
Here $\delta >0$ is chosen sufficiently small so that $\delta < (2m-d)/16$.

We make the bootstrap assumption
\begin{equation}\label{assm:apriori}
	\| \nu \|_{X_{T}} \le \epsilon_1.
\end{equation}

We establish the following a priori estimates in the next section.
\begin{proposition}\label{prop:bootstrap}
	Let $d=2,3$, and $m \in (d/2,1+2/d)$. Suppose $\nu_1 \in H^m_x L^2_y$ with $\| \nu_1 \|_{H^m_x L^2_y} \le \epsilon_0$. Let $\nu(t)$ be a local-in-time solution to \eqref{eqn:nu} on $[T^{-1}, 1]$, given by \Cref{lem:LWP}.
	If $\|\nu \|_{X_{T}} \le \epsilon_1$, then
	\[
	\| \nu \|_{X_T} \le C_0 \epsilon_0  + C_1 \epsilon_1^{1+2/d},
	\]
	where $C_0>0$ and $C_1>0$ are constants independent of $\epsilon_0$, $\epsilon_1$, and $T$.
\end{proposition}

\section{Global Existence with Dispersive Bounds}\label{sec:bootstrap}
In this section, we prove \Cref{prop:bootstrap}, based on the bootstrap argument.

\subsection{Estimates on the Nonlinear Terms}\label{sec:bounds-nonlinear}
We state and prove time decay estimates of the nonlinear terms under the bootstrap assumption.
Recall that $V_{\mathrm{H}}(\rho) = \lambda |x|^{-1} \ast \rho$ and $V_{\mathrm{xc}} (\rho ) = \mu\rho^{1/d}$.
To simplify the notation, we often write
$$V_{\nu}:=V(\rho_{\nu\nu^{\ast}}) \quad \text{and} \quad V_{j,\nu}:= V_j(\rho_{\nu\nu^{\ast}}) \quad \text{for}\quad j\in \{\mathrm{H},\mathrm{xc}\}.$$
We begin with estimates on $V_{\mathrm{H}}(\rho_{\nu\nu^{\ast}})$.

\begin{lemma} \label{lem:potential1v}
	Let $1 \le |\alpha| \le 2$. Under the bootstrap assumption \Cref{assm:apriori}, there holds
	\begin{align*}
		\| V_{\mathrm{H}}(\rho_{\nu\nu^{\ast}})(t)\|_{L^{\infty}} &\lesssim \epsilon_1^2 , &d=2,3,\\
		\| \partial_x^{\alpha} V_{\mathrm{H}}(\rho_{\nu\nu^{\ast}})(t) \|_{L^2} &\lesssim \epsilon_1^2  t^{- \delta(|\alpha|-1)/m}, &d=2,\\
		\| \partial_x^{\alpha} V_{\mathrm{H}}(\rho_{\nu\nu^{\ast}})(t) \|_{L^2} &\lesssim \epsilon_1^2 ,  &d=3,
	\end{align*}
	for $T^{-1} \le t \le 1$.
\end{lemma}

\begin{proof}
	By the definition of bootstrap norm, we have
	\[
	\| \nu(t) \|_{L^{\infty}_xL^2_y} \lesssim \epsilon_1 \quad\text{and}\quad \| \partial_x^{\beta} \nu(t) \|_{L^2_{x,y}} \lesssim \epsilon_1  t^{-\delta|\beta|/m}
	\]
	for $0\le |\beta| \le m$ and $T^{-1} \le t \le 1$. Applying Gagliardo--Nirenberg inequality, we get
	\[
	\| \partial_x \nu(t) \|_{L^q_x L^2_y} \lesssim \| \nu(t) \|_{\dot{H}^m_x L^2_y}^{1/m} \| \nu(t) \|_{L^r_x L^2_y}^{1-1/m} \quad\text{for}\quad \frac{1}{q} = \frac{1}{2m} + \frac{m-1}{mr}.
	\]
	Hence we obtain for $p\in[2,\infty]$ and $q\in[2,2m]$ that
	\[
	\|\nu(t) \|_{L^p_x L^2_y} \lesssim \epsilon_1 \quad\text{and}\quad \| \partial_x \nu(t) \|_{L^q_x L^2_y} \lesssim \epsilon_1  t^{-\delta/m}.
	\]
	We next derive the following bounds on the density:
	\begin{equation}\label{est:density-v}
		\| \rho_{\nu\nu^{\ast}}(t) \|_{L^p_x} \lesssim \epsilon_1^2 \quad\text{and}\quad \| \partial_x \rho_{\nu\nu^{\ast}}(t) \|_{L^q_x} \lesssim \epsilon_1^2 t^{-\delta/m}
	\end{equation}
	for $p \in [1,\infty]$ and $q \in [1,2m]$. Indeed, since
	\[
	\rho_{\nu\nu^{\ast}}(t,x) = \| \nu(t,x,\cdot) \|_{L^2_y}^2,
	\]
	we get for $p \in [1,\infty]$ that
	\[
	\| \rho_{\nu\nu^{\ast}}(t)\|_{L^p_x} \lesssim \| \nu(t)\|_{L^{2p}_x L^2_y}^2 \lesssim \epsilon_1^2.
	\]
	On the other hand, we have
	\[
	\partial_{x_j} \rho_{\nu\nu^{\ast}}(x) = \partial_{x_j} \langle \nu(t,x), \nu(t,x) \rangle_{L^2_y} = 2 \Re \langle \partial_{x_j} \nu(t,x), \nu(t,x) \rangle_{L^2_y}.
	\]
	Then we obtain for $q \in [1,2m]$ that
	\[
	\| \partial_x \rho_{\nu\nu^{\ast}} (t) \|_{L^q_x} \lesssim \| \nu(t) \|_{L^{r_1}_x L^2_y} \| \partial_x \nu(t) \|_{L^{r_2}_x L^2_y} \lesssim \epsilon_1^2 t^{-\delta/m}
	\]
	where $1/q= 1/r_1 + 1/r_2$ with $r_1 \in [2,\infty]$ and $r_2 \in [2,2m]$.
	
	We now estimate the potential $V_{\mathrm{H}}(\rho_{\nu\nu^{\ast}})= \lambda |\cdot|^{-1} \ast_x \rho_{\nu\nu^{\ast}}$.  For $R>0$, we compute
	\begin{equation*}
		V_{\mathrm{H}}(\rho_{\nu\nu^{\ast}})(t,x)
		\lesssim \int_{|x-z| \ge R} \frac{\rho_{\nu\nu^{\ast}}(t,z)}{|x-z|} \, \dd z + \int_{|x-z| \le R} \frac{\rho_{\nu\nu^{\ast}}(t,z)}{|x-z|} \, \dd z
		\lesssim R^{-1} \| \rho_{\nu\nu^{\ast}}(t) \|_{L^1} + R^{d-1}\| \rho_{\nu\nu^{\ast}}(t) \|_{L^{\infty}}.
	\end{equation*}
	Choose $R>0$ so that $R^{-1} \| \rho_{\nu\nu^{\ast}} (t) \|_{L^1} = R^{d-1} \| \rho_{\nu\nu^{\ast}}(t) \|_{L^{\infty}}$. This yields
	\begin{equation}\label{est:interpolation-riesz}
		\| V_{\mathrm{H}}(\rho_{\nu\nu^{\ast}}) (t) \|_{L^{\infty}} \lesssim \| \rho_{\nu\nu^{\ast}}(t) \|_{L^1}^{(d-1)/d} \| \rho_{\nu\nu^{\ast}} (t) \|_{L^{\infty}}^{1/d} \lesssim \epsilon_1^2 .
	\end{equation}
	It remains to bound $\partial_x^{\alpha} V_{\mathrm{H}}(\rho_{\nu\nu^{\ast}})$ for $1\le |\alpha| \le 2$. Observe that $\partial_x^{\alpha}V_{\mathrm{H}}(\rho_{\nu\nu^{\ast}}) = \partial_x^{\alpha} (-\Delta_x)^{-(d-1)/2} \rho_{\nu\nu^{\ast}}$.
	When $d=2$, we have $|\alpha| -d+1 \in [0,1]$. Since $\partial_x^{\alpha} (-\Delta_x)^{-|\alpha|/2}$ is a Calder\'on--Zygmund operator, we obtain that
	\[
	\| \partial_x^{\alpha} V_{\mathrm{H}}(\rho_{\nu\nu^{\ast}})(t) \|_{L^2}
	\sim \| (-\Delta_x)^{(|\alpha|-d+1)/2} \rho_{\nu\nu^{\ast}} (t) \|_{L^2}
	\lesssim \| \rho_{\nu\nu^{\ast}}(t) \|_{L^2}^{2-|\alpha|} \|\partial_x \rho_{\nu\nu^{\ast}}(t) \|_{L^2}^{|\alpha|-1} \lesssim \epsilon_1^2 t^{-\delta(|\alpha|-1)/m}.
	\]
	When $d=3$, we have $|\alpha| -d + 1 \in [-1, 0]$. We apply Hardy--Littlewood--Sobolev to get
	\begin{align*}
		\| \partial_x^{\alpha} V_{\mathrm{H}}(\rho_{\nu\nu^{\ast}})(t) \|_{L^2}
		\lesssim \| |\cdot|^{-1-|\alpha|} \ast_x \rho_{\nu\nu^{\ast}}(t) \|_{L^2}
		\lesssim \| \rho_{\nu\nu^{\ast}}(t) \|_{L^{p_1}}
		\lesssim \epsilon_1^2,
	\end{align*}
	where $1/p_1 = 1/2 + (2-|\alpha|)/3$. This concludes the proof of \Cref{lem:potential1v}.
\end{proof}

We need the following lemma to bound the nonlinear term that contains $V_{\mathrm{xc}}$, which is a generalization of \cite[Lemma~3.4]{Ginibre1994},  \cite[Lemma~2.3]{Hayashi1998a} to integral kernels.
\begin{lemma}\label{lem:powertype}
	Let $F:L^2_y(\mathbb{R}^d) \rightarrow L^2_y(\mathbb{R}^d)$ be a function defined by $F(\varphi) = \|\varphi\|_{L^2_y}^{2/d} \varphi$ for $\varphi \in L^2_y(\mathbb{R}^d)$. For any $\kappa \in (\dot{H}^m_x \cap L^{\infty}_x)(\mathbb{R}^d; L^2_y(\mathbb{R}^d))$, we have
	\[
	\| F(\kappa) \|_{\dot{H}^m_x L^2_y} \lesssim \| \kappa \|_{L^{\infty}_x L^2_y}^{2/d}\| \kappa \|_{\dot{H}^m_x L^2_y}
	\]
	for $m< 1 + 2/d$.
\end{lemma}

\begin{proof}
	We regard $L^2_y(\mathbb{R}^d;\mathbb{C})$ as a real Hilbert space with inner product
	\[
	\langle f, g \rangle := \Re \langle f, g \rangle_{L^2} =  \Re \int_{\mathbb{R}^d} f(y) \overline{g(y)} \, \dd y.
	\]
	We evaluate the G\^ateaux derivative $F(\varphi) = \|\varphi\|_{L^2_y}^{2/d} \varphi$ as
	\begin{align*}
		DF(\varphi)(\psi)
		=  \lim_{t \rightarrow 0} \frac{F(\varphi+t\psi)-F(\varphi)}{t}
		= \|\varphi \|_{L^2_y}^{2/d}\psi +  \frac{2}{d} \| \varphi \|_{L^2_y}^{2/d -2}\langle \varphi, \psi \rangle \varphi
	\end{align*}
	for $\varphi \neq 0$ in $L^2_y(\mathbb{R}^d)$. When $\varphi=0$, we have
	\[
	DF(0)(\psi) = \lim_{t \rightarrow 0} \frac{F(t\psi)-F(0)}{t} =0.
	\]
	It follows that
	\begin{equation}\label{est:DF-1}
		\| D F (\varphi_1) \|_{\mathrm{op}} \lesssim \| \varphi_1 \|_{L^2_y}^{2/d}.
	\end{equation}
	Moreover, we also claim that $DF \in C^{0,2/d}$, that is,
	\begin{equation}\label{est:DF-2}
		\| DF(\varphi_1) - DF(\varphi_2) \|_{\mathrm{op}} \lesssim \| \varphi_1 - \varphi_2 \|_{L^2_y}^{2/d}.
	\end{equation}
	If one of $\varphi_1$ and $\varphi_2$ is zero, this follows from the previous bound. Hence, we may assume that both are nonzero. Set $\varphi_j = r_j e_j$ with $r_j = \| \varphi_j\|_{L^2_y}$ and $e_j = \varphi_j / \| \varphi_j \|_{L^2_y}$ for $j=1,2$. For any $\psi$ with $\| \psi \|_{L^2_y} = 1$, we write
	\[
	DF(\varphi_j)(\psi) = r_j^{2/d} \left( \psi + \frac{2}{d} P_{e_j}\psi \right), \qquad P_{e_j}\psi =\langle e_j, \psi \rangle e_j.
	\]
	Therefore,
	\[
	\| DF(\varphi_1) - DF(\varphi_2) \|_{\mathrm{op}} \lesssim | r_1^{2/d} - r_2^{2/d} | + \|r_1^{2/d}P_{e_1} - r_2^{2/d} P_{e_2} \|_{\mathrm{op}}.
	\]
	Since $0< 2/d \le 1$, we have
	\[
	|r_1^{2/d} -r_2^{2/d}| \lesssim |r_1-r_2|^{2/d} \lesssim \| \varphi_1 - \varphi_2 \|_{L^2_y}^{2/d}.
	\]
	It remains to bound the projection term.
	We may assume $r_1 \le r_2$ without loss of generality. Then
	\[
	r_1^{2/d} P_{e_1} - r_2^{2/d} P_{e_2} = r_1^{2/d}(P_{e_1} - P_{e_2}) + (r_1^{2/d} - r_2^{2/d}) P_{e_2}.
	\]
	Using $\| P_{e_1} - P_{e_2} \|_{\mathrm{op}} \lesssim \| e_1 - e_2 \|_{L^2}$, we obtain that
	\[
	\| r_1^{2/d} P_{e_1} - r_2^{2/d} P_{e_2} \|_{\mathrm{op}}
	\lesssim r_1^{2/d} \|P_{e_1} - P_{e_2}\|_{\mathrm{op}} + |r_1^{2/d} - r_2^{2/d}| \lesssim r_1^{2/d} \|e_1 - e_2 \|_{L^2_y} + \| \varphi_1 - \varphi_2 \|_{L^2_y}^{2/d}.
	\]
	We observe that
	\begin{align*}
		r_1 \| e_1 - e_2 \|_{L^2_y}
		&\le \| r_1 e_1 - r_2 e_2 \|_{L^2_y} + \| ( r_1 -r_2)e_2 \|_{L^2_y} \\
		&\lesssim \| \varphi_1- \varphi_2 \|_{L^2_y} + |r_1 - r_2| \\
		&\lesssim \| \varphi_1 - \varphi_2 \|_{L^2_y}.
	\end{align*}
	We finally use $\| e_1 - e_2 \|_{L^2} \le 2$ and $\varphi_j = r_j e_j$ to bound
	\begin{align*}
		r_1^{2/d} \| e_1 - e_2 \|_{L^2_y}
		\lesssim (r_1 \| e_1 - e_2\|_{L^2_y})^{2/d} \| e_1 - e_2 \|_{L^2_y}^{1-2/d}
		\lesssim \| \varphi_1 - \varphi_2 \|_{L^2_y}^{2/d},
	\end{align*}
	which completes the proof of \eqref{est:DF-2}. We remark that \eqref{est:DF-1} and \eqref{est:DF-2} prove $F\in C^{1,2/d}$.
	
	Denote $\kappa_{\pm}= \kappa(x\pm h, \cdot)$, $\kappa= \kappa(x, \cdot)$ and $r= 1+2/d$ for notational convenience. Applying the mean value theorem for $F$, we obtain that
	\[
	F(\kappa_{\pm})- F(\kappa) = \int_0^1 DF(\lambda \kappa_{\pm} + (1-\lambda)\kappa) (\kappa_{\pm} - \kappa) \, \dd \lambda.
	\]
	It follows that
	\[
	F(\kappa_{+}) + F(\kappa_{-}) - 2 F(\kappa)
	= DF(\kappa) (\kappa_{+} + \kappa_{-} - 2\kappa) + \sum_{\pm}  \int_0^1 (DF(\lambda \kappa_{\pm} + (1-\lambda)\kappa) - DF(\kappa))(\kappa_{\pm} -\kappa) \, \dd \lambda.
	\]
	Estimates on $DF(\cdot)$ yield
	\begin{align*}
		\| F(\kappa_{+}) + F(\kappa_{-}) - 2F(\kappa)\|_{L^2_y}
		&\lesssim \|DF(\kappa)\|_{\mathrm{op}} \|\kappa_{+} + \kappa_{-} - 2\kappa\|_{L^2} + \sum_{\pm} \|\kappa_{\pm} - \kappa \|_{L^2_y}^r \\
		&\lesssim \| \kappa \|_{L^2_y}^{2/d} \| \kappa_{+} + \kappa_{-} - 2\kappa \|_{L^2_y} + \sum_{\pm} \| \kappa_{\pm} - \kappa\|_{L^2_y}^r.
	\end{align*}
	Since $r = 1 +2/d$, taking $L^2_x$ norm gives
	\begin{align*}
		\| F(\kappa_{+}) + F(\kappa_{-}) - 2 F(\kappa) \|_{L^2_{x,y}}
		\lesssim \| \kappa\|_{L^{\infty}_xL^2_y}^{2/d} \| \kappa_{+} + \kappa_{-} - 2\kappa \|_{L^2_{x,y}} + \sum_{\pm} \| \kappa_{\pm} - \kappa \|_{L^{2r}_x L^2_y}^{r}.
	\end{align*}
	Using the definitions of Besov seminorm, see \Cref{sec:notations}, we obtain that
	\begin{align*}
		\| F(\kappa)\|_{\dot{B}^{m}_{2,2}L^2_y}^{2}
		&= \int_0^{\infty} \sup_{|h| \le t } \|F(\kappa_{+}) + F(\kappa_{-}) - 2F(\kappa)\|_{L^2_{x,y}}^2 \frac{\dd t} {t^{1+2m}} \\
		&\lesssim \int_0^{\infty} \left( \| \kappa \|_{L^{\infty}_x L^2_y}^{4/d} \sup_{|h| \le t } \| \kappa_{+} + \kappa_{-} - 2\kappa \|_{L^2_{x,y}}^2 + \sum_{\pm} \sup_{|h| \le t } \| \kappa_{\pm} - \kappa\|_{L^{2r}_x L^2_y}^{2r} \right) \frac{\dd t} {t^{1+2m}} \\
		&\lesssim \| \kappa \|_{L^{\infty}_xL^2_y}^{4/d} \| \kappa\|_{\dot{B}^{m}_{2,2}L^2_y}^{2} + \| \kappa \|_{\dot{B}^{m/r}_{2r,2r}L^2_y}^{2r},
	\end{align*}
	noting that $0 <m/r < 1$.
	Applying the real interpolation theorem for vector-valued Besov spaces \cite[Section~5]{Amann1997}, we have
	\[
	\| \kappa \|_{\dot{B}^{m/r}_{2r,2r}L^2_y}^{2r} \lesssim \| \kappa \|_{\dot{B}^{m}_{2,2}L^2_y}^{2} \| \kappa \|_{\dot{B}^{0}_{\infty, \infty}L^2_y}^{2(r-1)}.
	\]
	Also, $\| \kappa \|_{\dot{B}^{0}_{\infty, \infty}L^2_y} \lesssim \| \kappa \|_{L^{\infty}_xL^2_y}$ and $r =1 + 2/d$ so that
	\[
	\| F (\kappa) \|_{\dot{B}^{m}_{2,2}L^2_y} \lesssim \| \kappa \|_{L^{\infty}_xL^2_y}^{2/d} \| \kappa \|_{\dot{B}^{m}_{2,2}L^2_y}.
	\]
	Noting that $\| \kappa \|_{\dot{B}^{m}_{2,2}L^2_y} \sim \| \kappa \|_{\dot{H}^m_x L^2_y}$, the proof of \Cref{lem:powertype} is complete.
\end{proof}

Combining \Cref{lem:potential1v} and \Cref{lem:powertype}, we obtain the following bound.
\begin{corollary}\label{coro:NL}
	Let $\nu(t)$ be a local-in-time solution to \eqref{eqn:nu} on $[T^{-1}, 1]$, given by \Cref{lem:LWP}. Under the bootstrap assumption \eqref{assm:apriori}, we have
	\[
	\| V(\rho_{\nu\nu^{\ast}})\nu(t) \|_{H^s_x L^2_y} \lesssim \epsilon_1^{1+2/d} t^{-\delta s/m}
	\]
	for $1 \le s \le m$ and $T^{-1} \le t \le 1$.
\end{corollary}

\begin{proof}
	We recall $V=V_{\mathrm{H}} + V_{\mathrm{xc}}$. We estimate each term separately.
	For the term with $V_{\mathrm{H}}(\rho_{\nu\nu^{\ast}})$, it follows from the fractional Leibniz rule and \Cref{lem:potential1v} that
	\begin{align*}
		\| V_{\mathrm{H}}(\rho_{\nu\nu^{\ast}}) \nu(t) \|_{H^s_x L^2_y}
		\lesssim \| V_{\mathrm{H}}(\rho_{\nu\nu^{\ast}})(t) \|_{L^{\infty}} \| \nu(t)\|_{H^s_x L^2_y}
		+ \| V_{\mathrm{H}}(\rho_{\nu\nu^{\ast}})(t) \|_{\dot{H}^s} \| \nu (t) \|_{L^{\infty}_xL^2_y}
		\lesssim \epsilon_1^3 t^{-\delta s/m}.
	\end{align*}
	For the term with $V_{\mathrm{xc}}(\rho_{\nu\nu^{\ast}})$, we apply \Cref{lem:powertype} to obtain
	\[
	\| V_{\mathrm{xc}}(\rho_{\nu\nu^{\ast}}) \nu(t) \|_{H^s_x L^2_y} \lesssim \| \nu(t) \|_{L^{\infty}_xL^2_y}^{2/d} \| \nu(t) \|_{H^s_x L^2_y}
	\lesssim \epsilon_1^{1+2/d} t^{-\delta s/m}.
	\]
	Combining these estimates, we obtain that
	\[
	\| V(\rho_{\nu\nu^{\ast}}) \nu(t) \|_{H^s_x L^2_y} \lesssim \epsilon_1^3 t^{-\delta s/m} + \epsilon_1^{1+2/d} t^{-\delta s/m} \lesssim \epsilon_1^{1+2/d} t^{-\delta s/m}.
	\]
\end{proof}

\subsection{Energy Estimates}
\begin{lemma}\label{lem:energy22}
	Let $\nu(t)$ be a local-in-time solution to \eqref{eqn:nu} on $[T^{-1}, 1]$, given by \Cref{lem:LWP}. Under the bootstrap assumption \Cref{assm:apriori}, there holds
	\[
	t^{\delta} \| \nu(t) \|_{H^m_x L^2_y} \lesssim \epsilon_0 + \epsilon_1^{1+2/d}
	\]
	for $T^{-1} \le t \le 1$.
\end{lemma}

\begin{proof}
	Recall that
	\[
	\nu(t,x,y)
	=  \ee^{\ii (t-1)\Delta_x/2} \nu_1(x,y)
	- \ii \int_1^t  \tau^{-1} \ee^{\ii (t-\tau)\Delta_x/2} V_{\nu} (\tau,x)  \nu(\tau,x,y) \, \dd \tau,
	\]
	where we recall $V_{\nu} = V(\rho_{\nu\nu^{\ast}})$. We bound
	\begin{equation}\label{eqn:energy22-2}
		t^{\delta}\| \nu(t) \|_{H^m_x L^2_y}
		\lesssim t^{\delta} \|  \nu_1 \|_{H^m_x L^2_y} +  t^{\delta}\int_t^1  \tau^{-1}  \| V_{\nu} \nu(\tau) \|_{H^m_x L^2_y}   \, \dd \tau.
	\end{equation}
	Applying \Cref{coro:NL}, we obtain that
	\[
	t^{\delta} \| \nu(t) \|_{H^m_x L^2_y}
	\lesssim \epsilon_0 +  t^{\delta} \int_t^1 \left(\epsilon_1^3 \tau^{-1-\delta} + \epsilon_1^{1+2/d} \tau^{-1-\delta} \right) \dd\tau \lesssim \epsilon_0 + \epsilon_1^{1+2/d},
	\]
	which completes the proof.
\end{proof}

\subsection{Remainder Estimates}\label{sec:remainder}

Consider the profile $\eta(t) =  \ee^{-\ii t\Delta_x/2} \nu(t)$. We can write \Cref{eqn:nu} as
\begin{equation}\label{eqn:f}
	\begin{aligned}
		\ii \partial_t \eta
		= \frac{1}{t} \ee^{-\ii t\Delta_x/2} V_{\nu} \nu
		=:  \frac{1}{t}  V_{\eta} \eta +R,
	\end{aligned}
\end{equation}
where we write
\begin{equation}\label{eqn:R}
	R(t)
	:= \frac{1}{t}\left( \ee^{-\ii t\Delta_x/2} V_{\nu} (t) \nu(t)  - V_{\eta} (t) \eta (t) \right),
\end{equation}
recalling that $V_{\nu} = V(\rho_{\nu\nu^{\ast}})$ and $ V_\eta=V(\rho_{\eta\eta^{\ast}})$.
We prove that $R(t)$ can be regarded as a remainder.
\begin{lemma}\label{lem:remainder2}
	Let $d=2,3$ and let $m$ satisfy $d/2<m < 1+2/d$. Choose $\theta$ such that $0 < \theta < (2m-d)/4$. Let $\nu(t)$ be a local-in-time solution to \eqref{eqn:nu} on $[T^{-1}, 1]$, given by \Cref{lem:LWP}. Then, for any $s$ with $d/2 <s \le m-2\theta$, under the bootstrap assumption \eqref{assm:apriori}, we have
	\begin{equation}\label{est:remainder2}
		\| R(t) \|_{(L^2_x \cap L^{\infty}_x)L^2_y} \lesssim \epsilon_1^{1+2/d} t^{-1+\theta-\delta(3s+2\theta)/m}
	\end{equation}
	for $T^{-1} \le t \le 1$.
\end{lemma}

\begin{proof}
	For $0 < \theta < (2m-d)/4$, we may choose $s$ satisfying $s> d/2$ and $s+2\theta \le m$. We will frequently use the estimate
	\[
	\| (\ee^{\pm \ii t\Delta_x /2} -1) u \|_{L^2_x} \lesssim t^{\theta} \left\| |\nabla_x|^{2\theta} u \right\|_{L^2_x},
	\]
	which is valid for $0 <\theta < 1$.
	Write $R= R_1 + R_2$, where
	\begin{align*}
		R_1(t) &= \frac{1}{t} ( \ee^{-\ii t\Delta_x/2} - 1 ) V_{\nu} (t) \nu(t) , \\
		R_2(t) &= \frac{1}{t} \left( V_\nu(t) \nu(t)- V_\eta(t) \eta(t) \right).
	\end{align*}
	
	We first estimate $R_1$.
	By Sobolev embedding and \Cref{coro:NL}, we get
	\begin{equation}\label{est:R1}
		\begin{aligned}
			\| R_1(t) \|_{(L^2_x \cap L^{\infty}_x)L^2_y}
			&\lesssim t^{-1} \left\| (\ee^{-\ii t\Delta_x /2} -1 ) V_{\nu} (t)\nu(t) \right\|_{H^s_x L^2_y}\\
			&\lesssim t^{-1+\theta} \| V_\nu(t) \nu(t)\|_{H^{s+2\theta}_x L^2_y}\\
			&\lesssim \epsilon_1^{1+2/d} t^{-1+ \theta -\delta(s+2\theta)/m}.
		\end{aligned}
	\end{equation}
	
	We now estimate $R_2$. For $j\in \{\mathrm{H},\mathrm{xc}\}$, we denote $F_{j,\nu}= V_{j,\nu}\nu$, where we recall that $V_{j,\nu}=V_j(\rho_{\nu\nu^*})$. We write $R_2 = R_{2,\mathrm{H}}+ R_{2,\mathrm{xc}}$, where
	\begin{align*}
		R_{2,j}(t) := \frac{1}{t} \Big( V_{j,\nu}(t) \nu(t) - V_{j, \eta}(t) \eta(t) \Big) = \frac{1}{t} \left( F_{j,\nu}(t) - F_{j,\eta}(t) \right)
	\end{align*}
	for $j\in \{\mathrm{H},\mathrm{xc}\}$.
	We claim for $s>d/2$ that
	\begin{equation}\label{est:F1v-F1f}
		\| F_{\mathrm{H},\nu} - F_{\mathrm{H},\eta}\|_{(L^2_x \cap L^{\infty}_x)L^2_y}
		\lesssim \| \nu\|_{H^s_x L^2_y}^2 \| \nu - \eta\|_{H^s_x L^2_y}
	\end{equation}
	and
	\begin{equation}\label{est:F2v-F2f}
		\| F_{\mathrm{xc},\nu} -F_{\mathrm{xc},\eta}\|_{(L^2_x \cap L^{\infty}_x)L^2_y}
		\lesssim \| \nu \|_{H^s_x L^2_y}^{2/d} \| \nu-\eta \|_{H^s_x L^2_y}.
	\end{equation}
	
	For the $F_{\mathrm{H}}$ term, using the definition $V_{\mathrm{H}} (\rho) = \lambda|\cdot|^{-1} \ast_x \rho$, we write
	\begin{equation}\label{eqn:F1v-F1f}
		\begin{aligned}
			F_{\mathrm{H},\nu} - F_{\mathrm{H},\eta}
			&= \lambda\left(|\cdot|^{-1} \ast \rho_{\nu\nu^{\ast}} \right) \nu - \lambda\left(|\cdot|^{-1} \ast \rho_{\eta\eta^{\ast}}\right) \eta\\
			&= \lambda\left(|\cdot|^{-1} \ast (\rho_{\nu\nu^{\ast}} - \rho_{\eta\eta^{\ast}} ) \right) \nu
			+ \lambda\left(|\cdot|^{-1} \ast \rho_{\eta\eta^{\ast}} \right) (\nu-\eta).
		\end{aligned}
	\end{equation}
	Recall from \Cref{est:interpolation-riesz} that
	\[
	\| |\cdot|^{-1} \ast \rho \|_{L^{\infty}(\mathbb{R}^d)} \lesssim\| \rho \|_{L^1(\mathbb{R}^d)}^{(d-1)/d} \| \rho \|_{L^{\infty}(\mathbb{R}^d)}^{1/d}
	\]
	for any $\rho \in (L^1 \cap L^{\infty})(\mathbb{R}^d)$.
	Moreover, for $p \in [1,\infty]$, using that $\eta(t)=\ee^{-\ii t\Delta_x/2}\nu(t)$ and that the group $\ee^{-\ii t\Delta/2}$ is unitary on $H^s$, we get
	\[
	\left\| \rho_{\nu\nu^{\ast}} - \rho_{\eta\eta^{\ast}} \right\|_{L^p_x} \lesssim ( \|\nu \|_{L^{2p}_x L^2_y} + \| \eta \|_{L^{2p}_x L^2_y} ) \| \nu-\eta\|_{L^{2p}_x L^2_y} \lesssim \| \nu \|_{H^s_x L^2_y} \| \nu-\eta \|_{H^s_x L^2_y},
	\]
	and similarly
	\[
	\left\| \rho_{\eta\eta^{\ast}} \right\|_{L^p_x} \lesssim \| \eta \|_{L^{2p}_x L^2_y}^2 \lesssim \| \eta \|_{H^s_x L^2_y}^2 = \| \nu \|_{H^s_x L^2_y}^2.
	\]
	Combining these bounds, we obtain from \eqref{eqn:F1v-F1f} that
	\begin{align*}
		&\| F_{\mathrm{H},\nu} - F_{\mathrm{H},\eta} \|_{(L^2_x \cap L^{\infty}_x)L^2_y} \\
		&\lesssim \left\| |\cdot|^{-1} \ast (\rho_{\nu\nu^{\ast}} - \rho_{\eta\eta^{\ast}}) \right\|_{L^{\infty}} \| \nu \|_{(L^2_x \cap L^{\infty}_x)L^2_y}
		+ \left\| |\cdot|^{-1} \ast \rho_{\eta\eta^{\ast}} \right\|_{L^{\infty}} \| \nu-\eta \|_{(L^2_x \cap L^{\infty}_x)L^2_y} \\
		&\lesssim \| \nu \|_{H^s_x L^2_y}^2 \| \nu-\eta \|_{H^s_x L^2_y},
	\end{align*}
	which proves \Cref{est:F1v-F1f}.
	
	For the $F_{\mathrm{xc}}$ term, we recall $V_{\mathrm{xc}}(\rho) = \mu\rho^{1/d}$, so that
	\[
	F_{\mathrm{xc},\nu} - F_{\mathrm{xc},\eta} = \mu\rho_{\nu\nu^{\ast}}^{1/d} \nu- \mu\rho_{\eta\eta^{\ast}}^{1/d} \eta.
	\]
	Since $F(\varphi) = \| \varphi\|_{L^2_y}^{2/d} \varphi$ satisfies $\| DF(\varphi) \|_{\mathrm{op}} \lesssim \| \varphi \|_{L^2}^{2/d}$, we apply the mean value theorem to get
	\[
	\|F_{\mathrm{xc},\nu} - F_{\mathrm{xc},\eta} \|_{L^2_y} \lesssim \left( \| \nu \|_{L^2_y}^{2/d} + \| \eta \|_{L^2_y}^{2/d} \right) \| \nu - \eta \|_{L^2_y}.
	\]
	Taking the $L^2_x \cap L^{\infty}_x$ norm, and using Sobolev embedding, it follows that
	\begin{align*}
		\| F_{\mathrm{xc},\nu} - F_{\mathrm{xc},\eta} \|_{(L^2_x \cap L^{\infty}_x)L^2_y}
		&\lesssim \left( \| \nu \|_{L^{\infty}_x L^2_y}^{2/d} + \| \eta \|_{L^{\infty}_x L^2_y}^{2/d} \right) \| \nu - \eta \|_{(L^2_x \cap L^{\infty}_x)L^2_y}
		\lesssim \| \nu \|_{H^s_x L^2_y}^{2/d} \|\nu-\eta \|_{H^s_x L^2_y}.
	\end{align*}
	This yields \Cref{est:F2v-F2f}.
	
	Using \Cref{est:F1v-F1f}, \Cref{est:F2v-F2f}, and the bootstrap assumption \Cref{assm:apriori}, we may bound $R_2$ by
	\begin{equation}\label{est:R2}
		\begin{aligned}
			\| R_{2}(t) \|_{(L^2_x \cap L^{\infty}_x)L^2_y}
			&\lesssim t^{-1} \left( \| \nu(t)\|_{H^s_x L^2_y}^2 + \| \nu(t)\|_{H^s_x L^2_y}^{2/d} \right) \| (1-\ee^{-\ii t\Delta/2})\nu(t) \|_{H^s_x L^2_y} \\
			&\lesssim t^{-1+\theta} \left( \| \nu(t)\|_{H^s_x L^2_y}^2 + \| \nu(t)\|_{H^s_x L^2_y}^{2/d} \right) \| \nu(t) \|_{H^{s+2\theta}_x L^2_y} \\
			&\lesssim \epsilon_1^{1+2/d} t^{-1 + \theta -\delta (3s+2\theta)/m}.
		\end{aligned}
	\end{equation}
	Combining \Cref{est:R1} and \Cref{est:R2}, we obtain \Cref{est:remainder2}, completing the proof of \Cref{lem:remainder2}.
\end{proof}

Now we are ready to prove the following lemma.
\begin{lemma}\label{lem:decay}
	Let $\nu(t)$ be a local-in-time solution to \eqref{eqn:nu} on $[T^{-1}, 1]$, given by \Cref{lem:LWP}. Under the bootstrap assumption \eqref{assm:apriori}, we have
	\begin{equation}\label{claim:v}
		\| \nu(t) \|_{(L^2_x \cap L^{\infty}_x) L^2_y}  \lesssim \epsilon_0 + \epsilon_1^{1+2/d}
	\end{equation}
	for $T^{-1} \le t \le 1$.
\end{lemma}

\begin{proof}
	The $L^2_{x,y}$ bound follows from conservation of mass. We focus on the $L^{\infty}_x L^2_y$ bound. Recall \eqref{eqn:f} and \eqref{eqn:R}.
	We introduce the modified profile $\widetilde{\eta}(t)= \ee^{-\ii \Psi(t)}\eta(t)$ where
	\begin{equation}\label{Def:Psi}
		\Psi(t,x) = \int_{t}^{1} \tau^{-1} V_{\eta} (\tau,x) \, \dd \tau \in \mathbb{R},
	\end{equation}
	where $V_{\eta} = V(\rho_{\eta\eta^{\ast}})$. Since $|\ee^{-\ii \Psi (t)}| =1$, this is a phase correction term.
	Then \Cref{eqn:f} becomes
	\[
	\ii \partial_t \widetilde{\eta}(t)
	= \ee^{-\ii \Psi(t)}(\ii \partial_t \eta(t) + \partial_t \Psi (t) \eta(t) )
	= \ee^{-\ii \Psi(t)} R(t)
	\]
	so that
	\begin{equation}\label{eqn:tilde-f}
		\widetilde{\eta}(t) = \widetilde{\eta}(1) - \ii \int_{1}^{t} \ee^{-\ii \Psi(\tau)} R(\tau) \, \dd \tau.
	\end{equation}
	Choose $\theta$ satisfying $\delta < \theta < (2m-d)/4$, taking $\delta$ smaller if necessary. Choose $s$ with $d/2 < s \le m-2\theta$. Noting that $\nu(t) = \eta(t) + (1- \ee^{-\ii t\Delta/2}) \nu(t)$, and $|\eta(t)| = |\widetilde{\eta}(t)|$,  we may bound
	\begin{align*}
		\| \nu(t) \|_{L^{\infty}_x L^2_y}
		&\lesssim \| \eta(t) \|_{L^{\infty}_x L^2_y}  + \| (1- \ee^{-\ii t \Delta/2}) \nu(t)\|_{H^s_x L^2_y} \\
		&\lesssim \| \widetilde{\eta}(t) \|_{L^{\infty}_x L^2_y}  + t^\theta \| \nu(t) \|_{H^{s+2\theta}_x L^2_y},
	\end{align*}
	for $T^{-1} \le t \le 1$, upon using Sobolev embedding.
	Using \Cref{eqn:tilde-f} and \Cref{lem:remainder2}, the first term on the right hand side can be estimated by
	\begin{align*}
		\|\widetilde{\eta}(t)\|_{L^{\infty}_x L^2_y}
		&\lesssim \| \widetilde{\eta}(1)\|_{L^{\infty}_x L^2_y}
		+ \int_t^1 \| R(\tau) \|_{L^{\infty}_x L^2_y} \, \dd \tau \\
		&\lesssim \epsilon_0 + \epsilon_1^{1+2/d}\int_t^1 \tau^{-1+\theta - \delta(3s+2\theta)/m } \, \dd \tau \\
		&\lesssim \epsilon_0 + \epsilon_1^{1+2/d}.
	\end{align*}
	Using \Cref{lem:energy22}, the other term can be estimated by
	\[
	t^{\theta} \| \nu(t)\|_{H^{s+2\theta}_x L^2_y} \lesssim t^{\theta-\delta} (\epsilon_0 + \epsilon_1^{1+2/d}) \lesssim \epsilon_0 + \epsilon_1^{1+2/d},
	\]
	noting that $\theta > \delta$. Combining these two bounds, we obtain \Cref{claim:v}.
\end{proof}

\subsection{Proof of Proposition \ref{prop:bootstrap}}
Now we are ready to complete the proof of \Cref{prop:bootstrap}.
\begin{proof}[Proof of \Cref{prop:bootstrap}]
	Combining \Cref{lem:energy22} and \Cref{lem:decay}, we get
	\[
	\sup_{t \in [T^{-1},1]} \left\{ \|\nu(t)\|_{L^2_{x,y}} + \| \nu(t) \|_{L^{\infty}_x L^2_y} + t^{\delta} \| \nu(t) \|_{H^m_x L^2_y} \right\} \le C_0 \epsilon_0  + C_1 \epsilon_1^{1+2/d}
	\]
	for some $C_0>0$ and $C_1 >0$, independent of $T$. This yields \Cref{prop:bootstrap}, recalling the definition \eqref{def:bootstrapnorm} of $X_T$ norm.
\end{proof}

\begin{remark}
	Recalling that $\nu= \mathcal{I}\kappa$ is the pseudo-conformal transform of $\kappa$, this in particular proves that \eqref{eqn:KS-kappa} has a unique global-in-time solution $\kappa(t)$ on the time interval $[1,\infty)$ satisfying
	\[
	\| \kappa(t) \|_{L^p_x L^2_y} \lesssim \epsilon_0 t^{-d(1/2-1/p)}, \qquad \| \langle \mathrm{J}_x \rangle^{m} \kappa(t) \|_{L^2_{x,y}} \lesssim \epsilon_0 t^{\delta}
	\]
	for $2 \le p \le \infty$ and $t \ge 1$.
\end{remark}

\section{Modified Scattering}\label{sec:modsc}
In this section, we prove modified scattering, completing the proof of \Cref{thm:main2} and \Cref{thm:main1}.

\subsection{Proof of Theorem \ref{thm:main2}}

Choose $\epsilon_1 = 4C_0 \epsilon_0$.  Recall the bootstrap assumption \eqref{assm:apriori} given by
\[
\| \nu \|_{X_{T}} \le \epsilon_1.
\]
Taking $\epsilon_0 >0$ sufficiently small, we may assume that
\[
C_1  \epsilon_1^{1+2/d} \le \frac{\epsilon_1}{4}.
\]
Then \Cref{prop:bootstrap} yields
\[
\| \nu \|_{X_T} \le C_0 \epsilon_0 + C_1 \epsilon_1^{1+2/d} \le \frac{\epsilon_1}{2}.
\]
By the local well-posedness (\Cref{lem:LWP}) and a standard continuation argument, this improved estimate closes the bootstrap argument. Consequently, the solution $\nu$ is constructed on $(0,1]$ and $\| \nu \|_{X_T} \le \epsilon_1$ holds for all $T\ge 1$.

The corresponding solution $\kappa$ on $[1,\infty)$ is obtained via the inverse of pseudo-conformal transformation. Recalling \eqref{eqn:pseudo-conformal-prop} and \eqref{def:bootstrapnorm}, this also proves
\[
t^{-\delta}\| \langle \mathrm{J}_x\rangle^m \kappa(t) \|_{L^2_{x,y}} + t^{d(1/2-1/p)}\| \kappa(t) \|_{L^p_x L^2_y} \lesssim \epsilon_0
\]
for $2\le p \le \infty$ and $t \ge 1$, which in particular establishes the time decay estimate \Cref{est:timedecay}.

Next, we establish modified scattering for $\kappa$. We first prove that $\widetilde{\eta}(t)$ has a limit as $t \rightarrow 0^{+}$ in $(L^2_x \cap L^{\infty}_x)L^2_y$.
Taking $\delta>0$ smaller if necessary, we choose
\[
4\delta < \theta < (2m-d)/4, \qquad
d/2 < s \le m-2\theta.
\]
Then $(3s+2\theta)/m \le 3$, so \Cref{lem:remainder2} gives
$$
\| R(t) \|_{L^p_x L^2_y} \lesssim \epsilon_1^{1+2/d} t^{-1 +\theta - \delta(3s + 2\theta)/m} \lesssim \epsilon_1^{1+2/d} t^{-1+\theta - 3\delta}
$$
for $p \in [2,\infty]$ and $0< t \le 1$. Since $\theta - 3 \delta >0$, this bound is integrable at $t=0$. Recalling \eqref{eqn:tilde-f}, we get
\begin{equation}\label{est:tilde-f}
	\| \widetilde{\eta}(t_1) - \widetilde{\eta}(t_2) \|_{(L^2_x \cap L^{\infty}_x)L^2_y} \lesssim \int_{t_2}^{t_1} \| R(\tau) \|_{(L^2_x \cap L^{\infty}_x)L^2_y} \, \dd \tau \lesssim \epsilon_1^{1+2/d} \left(t_1^{\theta - 3\delta } - t_2^{\theta - 3\delta} \right)
\end{equation}
for $0< t_2 \le t_1 \le 1$.
Therefore, the limit of $\widetilde{\eta}(t)$ as $t \rightarrow 0^{+}$ is well-defined in $(L^2_x \cap L^{\infty}_x)L^2_y$, and we denote it by $\widetilde{\eta}(0)$.
It remains to undo the conjugations and pseudo-conformal transformation to obtain the modified scattering for $\kappa$. Recalling that
\[
\widetilde{\eta}(t)
= \ee^{-\ii \Psi(t)} \eta(t)
= \ee^{-\ii \Psi(t)} \ee^{-\ii t\Delta_x/2} \nu(t),
\]
where $\Psi(t)$ is defined in \Cref{Def:Psi}, we obtain
\begin{align*}
	\omega(t) &= \overline{\nu(1/t)}
	= \ee^{-\ii \Delta_x/(2t)} \ee^{-\ii \Psi(1/t)}  \overline{\widetilde{\eta}(1/t)}
	= \mathcal{F} \mathrm{M}(t) \mathcal{F}^{-1} \ee^{-\ii \Psi(1/t)} \overline{\widetilde{\eta}(1/t)}.
\end{align*}
Then we may write
\begin{align*}
	\kappa(t)
	&= \mathrm{M}(t) \mathrm{D}(t) \omega(t)  \\
	&= \mathrm{M}(t) \mathrm{D}(t) \mathcal{F} \mathrm{M}(t) \mathcal{F}^{-1} \ee^{-\ii \Psi(1/t)} \overline{\widetilde{\eta}(1/t)} \\
	&= \ee^{\ii t\Delta_x/2} \mathcal{F}^{-1} \ee^{-\ii \Psi(1/t)} \overline{\widetilde{\eta}(1/t)} \\
	&= \ee^{\ii t\Delta_x/2} \ee^{-\ii \Psi(1/t, -\ii \nabla_x)} \mathcal{F}^{-1}\overline{\widetilde{\eta}(1/t)}.
\end{align*}
Since $\ee^{\ii t \Delta/2}$ is unitary on $L^2_x$, we obtain from \Cref{est:tilde-f} that
\begin{equation}\label{est:modsca-u-1}
	\| \kappa(t) - \ee^{\ii t\Delta_x/2} \ee^{-\ii \Psi(1/t, -\ii \nabla_x)} \mathcal{F}^{-1}\overline{\widetilde{\eta}(0)}\|_{L^2_{x,y}}
	\lesssim \| \widetilde{\eta}(1/t) - \widetilde{\eta}(0) \|_{L^2_{x,y}}
	\lesssim \epsilon_1^{1+2/d} t^{-\theta+3\delta}
\end{equation}
for $t \ge 1$.

Next, since the phase correction does not change the density, we have
\[
\rho_{\widetilde{\eta}\widetilde{\eta}^{\ast}}(t,x) = \int |\widetilde{\eta}(t,x,y)|^2 \, \dd y = \int |\eta (t,x,y)|^2 \, \dd y = \rho_{\eta \eta^{\ast}}(t,x).
\]
Therefore, $V_{\eta}(t) = V_{\widetilde{\eta}}(t)$ for $0 < t \le 1$, where $V_{\nu} = V(\rho_{\nu \nu^{\ast}})$. Moreover, recalling that $\widetilde{\eta}(0)$ is well-defined, we set
\[
V_{\widetilde{\eta}}(0) := V(\rho_{\widetilde{\eta}(0)\widetilde{\eta}(0)^{\ast}} ).
\]
For $t \ge 1$, we decompose
\begin{equation}\label{eqn:Psi-decomp}
	\begin{aligned}
		\Psi(t^{-1}, k)
		&= \int_{1/t}^1 \tau^{-1} V_{\widetilde{\eta}} (\tau,k) \, \dd \tau
		= \int_1^t s^{-1} V_{\widetilde{\eta}} (s^{-1},k) \, \dd s \\
		&= V_{\widetilde{\eta}}(0,k) \log t+ \int_1^t s^{-1}\left( V_{\widetilde{\eta}}(s^{-1}, k) - V_{\widetilde{\eta}}(0,k) \right) \dd s \\
		&=: g_{\infty}(k) \log t + \Psi^{r}(t^{-1}, k),
	\end{aligned}
\end{equation}
where we set
\begin{align*}
	g_{\infty}(k) := V_{\widetilde{\eta}}(0,k), \qquad
	\Psi^r (t^{-1}, k) := \int_1^{t} s^{-1} \left( V_{\widetilde{\eta}}(s^{-1}, k) - V_{\widetilde{\eta}}(0,k) \right) \dd s.
\end{align*}
We also introduce
\[
\Psi^r (0,k) := \int_1^{\infty} s^{-1} \left( V_{\widetilde{\eta}}(s^{-1}, k) - V_{\widetilde{\eta}}(0,k) \right) \dd s,
\]
provided that the improper integral converges.

To prove the convergence, we bound $V_{\widetilde{\eta}}(t) - V_{\widetilde{\eta}}(0)$ for $0 < t \le 1$. We first use \Cref{est:tilde-f} to obtain
\begin{equation}\label{est:AC-1}
\begin{aligned}
    \| \widetilde{\eta}(t) - \widetilde{\eta}(0) \|_{(L^2_x \cap L^{\infty}_x) L^2_y}
    \lesssim
    \epsilon_1^{1+2/d} t^{\theta-3\delta}
\end{aligned}
\end{equation}
for $0 < t \le 1$. Moreover, Sobolev embedding and \eqref{est:tilde-f} yield
\begin{equation}
\begin{aligned}
    \| \widetilde{\eta}(t) \|_{(L^2_x \cap L^{\infty}_x) L^2_y}
    &\lesssim
    \| \widetilde{\eta}(1) \|_{(L^2_x \cap L^{\infty}_x) L^2_y}
    +
    \| \widetilde{\eta}(t) - \widetilde{\eta}(1) \|_{(L^2_x \cap L^{\infty}_x) L^2_y}\\
    &\lesssim
    \| \nu_1 \|_{H^m_x L^2_y} 
    +
    \| \widetilde{\eta}(t) - \widetilde{\eta}(1) \|_{(L^2_x \cap L^{\infty}_x) L^2_y} \\
    &\lesssim
    \epsilon_0 + \epsilon_1^{1+2/d}\\
    &\lesssim
    \epsilon_1
\end{aligned}
\end{equation}
for $ 0 < t \le 1$. Recalling that
\[
    \rho_{\widetilde{\eta}\widetilde{\eta}^{\ast}}(t,x) = \| \widetilde{\eta}(t,x,\cdot) \|_{L^2_y}^2,
\]
we have
\begin{equation}\label{est:AC-2}
    \| \rho_{\widetilde{\eta}\widetilde{\eta}^{\ast}}(t) - \rho_{\widetilde{\eta}\widetilde{\eta}^{\ast}}(0) \|_{L^p_x}
    \lesssim 
    \left( \| \widetilde{\eta}(t) \|_{L^{2p}_x L^2_y} + \| \widetilde{\eta}(0) \|_{L^{2p}_x L^2_y}  \right)
    \| \widetilde{\eta}(t) - \widetilde{\eta}(0) \|_{L^{2p}_x L^2_y}
    \lesssim
    \epsilon_1^{2+2/d} t^{\theta - 3\delta}
\end{equation}
uniformly for $ 1 \le p \le \infty$ and $ 0 < t \le 1$.

Recall the decomposition $V= V_{\mathrm{H}} + V_{\mathrm{xc}}$. We apply \Cref{est:AC-2} for $p=1$ and $p = \infty$ to get
\begin{align*}
	\| V_{\mathrm{H},\widetilde{\eta}}(t) - V_{\mathrm{H},\widetilde{\eta}}(0) \|_{L^{\infty}}
	&\lesssim \left\| |\cdot|^{-1} \ast (\rho_{\widetilde{\eta}\widetilde{\eta}^{\ast}}(t) -\rho_{\widetilde{\eta}\widetilde{\eta}^{\ast}}(0))  \right\|_{L^{\infty}}\\
	&\lesssim \left\| \rho_{\widetilde{\eta}\widetilde{\eta}^{\ast}}(t) -\rho_{\widetilde{\eta}\widetilde{\eta}^{\ast}}(0) \right\|_{L^1}^{(d-1)/d} \left\| \rho_{\widetilde{\eta}\widetilde{\eta}^{\ast}}(t) -\rho_{\widetilde{\eta}\widetilde{\eta}^{\ast}}(0)\right\|_{L^{\infty}}^{1/d} \\
	&\lesssim \epsilon_1^{2+2/d} t^{\theta -3 \delta}.
\end{align*}
In addition, using $|a^{2/d} - b^{2/d} | \le |a-b|^{2/d}$ for $a,b\ge0$ and \eqref{est:tilde-f}, we obtain that
\begin{equation*}
\begin{aligned}
\| V_{\mathrm{xc},\widetilde{\eta}}(t) - V_{\mathrm{xc},\widetilde{\eta}}(0) \|_{L^{\infty}} 
&\lesssim 
\left\| \rho_{\widetilde{\eta}\widetilde{\eta}^{\ast}}^{1/d}(t) - \rho_{\widetilde{\eta}\widetilde{\eta}^{\ast}}^{1/d}(0) 
\right\|_{L^{\infty}} 
\lesssim
\| \widetilde{\eta}(t) - \widetilde{\eta}(0) \|_{L^{\infty}_x L^2_y}^{2/d}
\lesssim 
\epsilon_1^{2(1+2/d)/d}t^{2(\theta-3\delta)/d}.
\end{aligned}
\end{equation*}
Combining these two bounds, we obtain for $t \ge 1$ that
\[
\| V_{\widetilde{\eta}}(1/t) -  V_{\widetilde{\eta}}(0)\|_{L^{\infty}}
\lesssim \| V_{\mathrm{H},\widetilde{\eta}}(1/t) -  V_{\mathrm{H},\widetilde{\eta}}(0)\|_{L^{\infty}}
+ \| V_{\mathrm{xc},\widetilde{\eta}}(1/t) -  V_{\mathrm{xc},\widetilde{\eta}}(0)\|_{L^{\infty}}
\lesssim \epsilon_1^{2(1+2/d)/d}t^{-\delta'}
\]
where $\delta' := \min \{ \theta - 4\delta, \, 2(\theta - 3\delta)/d \} >0$, noting that $\epsilon_1 \le 1$ and $\theta > 4\delta$.
In particular, 
\[
    \int_1^{\infty} s^{-1} \left\| V_{\widetilde{\eta}}(s^{-1}) - V_{\widetilde{\eta}}(0) \right\|_{L^{\infty}_k} \dd s
    \lesssim
    \epsilon_1^{2(1+2/d)/d} \int_1^{\infty} s^{-1-\delta'} \, \dd s
    <\infty,
\]
which proves that $\Psi^r (0,k)$ is well-defined.
Now using $|\ee^{-\ii a} - \ee^{-\ii b}| \le |a-b|$ for $a, b \in \mathbb{R}$, we bound
\begin{equation}\label{est:modsca-u-2}
	\begin{aligned}
		\| \ee^{-\ii \Psi^r (1/t, -\ii \nabla_x)} - \ee^{-\ii \Psi^r (0, -\ii \nabla_x)} \|_{\mathrm{op}}
		&\lesssim \| \ee^{-\ii \Psi^{r}(1/t)} - \ee^{-\ii \Psi^r (0)} \|_{L^{\infty}_k} \\
		&\lesssim  \|\Psi^r (1/t) - \Psi^r (0) \|_{L^{\infty}_k}\\
		&\lesssim \int_t^{\infty} s^{-1} \| V_{\widetilde{\eta}}(s^{-1}) - V_{\widetilde{\eta}} (0) \|_{L^{\infty}_k} \, \dd s\\
		& \lesssim \epsilon_1^{2(1+2/d)/d} t^{-\delta'}.
	\end{aligned}
\end{equation}
Using \Cref{est:modsca-u-1}, \Cref{eqn:Psi-decomp}, \Cref{est:modsca-u-2}, and the decomposition
\begin{align*}
	\ee^{-\ii \Psi(1/t, -\ii \nabla_x)} - \ee^{-\ii g_{\infty}(-\ii \nabla_x)\log t} \ee^{-\ii \Psi^r (0, -\ii \nabla_x)}
	= \ee^{-\ii g_{\infty}(-\ii \nabla_x) \log t} (\ee^{-\ii \Psi^r (1/t, -\ii \nabla_x)} - \ee^{-\ii \Psi^r (0, -\ii \nabla_x)}),
\end{align*}
it follows that
\begin{align*}
	&\| \kappa(t) - \ee^{-\ii(-t\Delta_x/2 + g_{\infty} (-\ii \nabla ) \log t)} \kappa_{\infty}  \|_{L^2_{x,y}}\\
	&\lesssim \| \kappa(t) - \ee^{\ii t\Delta_x/2} \ee^{-\ii \Psi(1/t, -\ii \nabla_x)} \mathcal{F}^{-1}\overline{\widetilde{\eta}(0)}\|_{L^2_{x,y}} \\
	& \qquad + \| \ee^{\ii t\Delta_x/2}(\ee^{-\ii \Psi(1/t, -\ii \nabla_x)} - \ee^{-\ii g_{\infty}(-\ii \nabla_x)\log t} \ee^{-\ii \Psi^r (0, -\ii \nabla_x)})\mathcal{F}^{-1} \overline{\widetilde{\eta}(0)} \|_{L^2_{x,y}}  \\
	&\lesssim \epsilon_1^{1+2/d} t^{-\theta +3\delta} + \epsilon_1^{2(1+2/d)/d} t^{-\delta'} \\
	&\lesssim \epsilon_1^{2(1+2/d)/d} t^{-\delta'}
\end{align*}
for $t \ge 1$, where $\kappa_{\infty}:= \ee^{-\ii \Psi^r(0,-\ii \nabla_x)} \mathcal{F}^{-1} \overline{\widetilde{\eta}(0)} \in L^2_{x,y}$.

\subsection{Proof of Theorem \ref{thm:main1}}\label{sec:proof-main1}
In this section, we complete the proof of \Cref{thm:main1}.
\begin{proof}[Proof of \Cref{thm:main1}]
	Let $\gamma_1 \in \mathfrak{S}^1$ be a density matrix satisfying
	\[
	\| \langle \mathrm{J}(1) \rangle^m \gamma_1 \langle \mathrm{J}(1) \rangle^m \|_{\mathfrak{S}^1} + \| \rho_{\gamma_1}\|_{L^{\infty}_x} \le \epsilon_0.
	\]
	Define $\kappa_1:= \gamma_1^{1/2}$. Then
	\[
	\| \langle \mathrm{J}_x(1) \rangle^m \kappa_1 \|_{L^2_{x,y}} + \|  \kappa_1 \|_{L^\infty_x L^2_y} \lesssim \epsilon_0^{1/2}.
	\]
	Hence, for sufficiently small $\epsilon_0 >0$, \Cref{thm:main2} gives a global-in-time solution $\kappa$ of \Cref{eqn:KS-kappa-kernel} satisfying
	\[
	\langle \mathrm{J}_x \rangle^m \kappa \in C([1,\infty);\mathfrak{S}^2), \qquad \| \kappa(t) \|_{L^p_x L^2_y} \lesssim \epsilon_0^{1/2} t^{-d(1/2-1/p)}
	\]
	for $ 2 \le p \le \infty$ and $t \ge 1$.
	Now define $\gamma(t) = \kappa(t) \kappa(t)^{\ast}$ for $t \ge 1$. Then $\gamma(t)$ is nonnegative trace class.  Moreover, for each $t \ge 1$, we have
	\[
	\| \langle\mathrm{J}(t) \rangle^m \gamma(t)\langle\mathrm{J}(t) \rangle^m\|_{\mathfrak{S}^1} = \| \langle\mathrm{J}(t) \rangle^m \kappa (t)\|_{\mathfrak{S}^2}^2, \qquad
	\| \rho_{\gamma}(t) \|_{L^p} = \| \kappa (t) \|_{L^{2p}_xL^2_y}^2.
	\]
	Therefore, we get
	\[
	\langle \mathrm{J} \rangle^m \gamma \langle \mathrm{J} \rangle^m \in C([1,\infty), \mathfrak{S}^1), \qquad \| \rho_{\gamma}(t) \|_{L^p} \lesssim \epsilon_0 t^{-d(1-1/p)}
	\]
	for $1 \le p \le \infty$ and $t \ge 1$.
	
	Let $H_{\gamma} = -\Delta/2 + V(\rho_{\gamma})$ be the Hamiltonian which is self-adjoint. Since $ \rho_{\gamma} = \rho_{\kappa \kappa^{\ast}}$, we compute
	\begin{align*}
		\ii \partial_t \gamma
		&= (\ii \partial_t \kappa) \kappa^{\ast} + \kappa (\ii \partial_t \kappa^{\ast}) \\
		&= (\ii \partial_t \kappa) \kappa^{\ast} + \kappa (-\ii \partial_t \kappa)^{\ast} \\
		&= H_{\gamma} \kappa \kappa^{\ast} - \kappa  \kappa^{\ast} H_{\gamma} \\
		&= [H_{\gamma}, \gamma].
	\end{align*}
	Therefore, $\gamma(t)$ solves \Cref{eqn:KS-gamma}.
	
	It remains to prove modified scattering for $\gamma$. Suppose there exists $\delta'>0$, $\kappa_{\infty} \in \mathfrak{S}^2$ and a real-valued function $g_{\infty}$
	such that
	\[
	\|  \kappa(t) - \ee^{-\ii(-t\Delta/2 + g_{\infty}(-\ii \nabla) \log t)} \kappa_{\infty} \|_{\mathfrak{S}^2} \lesssim \epsilon_0^{1/2} t^{-\delta'}
	\]
	holds for all $t \ge 1$. Using $AA^{\ast} - BB^{\ast} = (A-B)A^{\ast} + B(A-B)^{\ast}$ and $\gamma(t) = \kappa(t) \kappa(t)^{\ast}$, we obtain that
	\[
	\| \gamma(t)- \ee^{-\ii(-t\Delta/2 + g_{\infty} (-\ii \nabla ) \log t)} \gamma_{\infty} \ee^{\ii (-t\Delta/2 + g_{\infty} (-\ii \nabla ) \log t)} \|_{\mathfrak{S}^1}
	\lesssim \epsilon_0 t^{-\delta'}
	\]
	for all $t \ge 1$, where $\gamma_{\infty}:= \kappa_{\infty} \kappa_{\infty}^{\ast} \in \mathfrak{S}^1$, upon taking smaller $\delta' >0$ if necessary.
	
	Finally, if there is another $\widetilde{\gamma}_{\infty}$ satisfying
	\[
	\| \gamma(t)- \ee^{-\ii(-t\Delta/2 + g_{\infty} (-\ii \nabla ) \log t)} \widetilde{\gamma}_{\infty} \ee^{\ii (-t\Delta/2 + g_{\infty} (-\ii \nabla ) \log t)} \|_{\mathfrak{S}^1} \rightarrow 0
	\]
	as $t \rightarrow \infty$, then $\gamma_{\infty} = \widetilde{\gamma}_{\infty}$ since
	\[
	\| \gamma_{\infty} - \widetilde{\gamma}_{\infty} \|_{\mathfrak{S}^1}
	= \|\ee^{-\ii(-t\Delta/2 + g_{\infty} (-\ii \nabla ) \log t)} (\gamma_{\infty}-\widetilde{\gamma}_{\infty}) \ee^{\ii (-t\Delta/2 + g_{\infty} (-\ii \nabla ) \log t)} \|_{\mathfrak{S}^1} \rightarrow 0
	\]
	as $t \rightarrow \infty$. This proves the uniqueness of modified scattering state, which concludes the proof of \Cref{thm:main1}.
\end{proof}

\appendix
\section{Local Well-Posedness}\label{sec:LWP}
In this appendix, we prove \Cref{lem:LWP}.
\begin{proof}[Proof of \Cref{lem:LWP}]
	Set $R= 2 \| \nu_1 \|_{H^m_x L^2_y}$. We shall choose $T= T(R) > 1$, sufficiently close to $1$, to be determined later. Define the map
	\[
	\Phi(\nu)(t,x,y)
	:=  \ee^{\ii (t-1)\Delta_x/2} \nu_1(x,y)
	- \ii \int_1^t  \ee^{\ii (t-\tau)\Delta_x/2} \tau^{-1} V_{\nu} (\tau,x)  \nu(\tau,x,y) \, \dd \tau
	\]
	for $t \in [T^{-1},1]$, recalling that $V_\nu = V(\rho_{\nu\nu^{\ast}})$. Consider the following complete metric space $(Z_{T,R}, d)$ where
	\[
	Z_{T,R} := \left\{ \nu \in C([T^{-1},1], L^2_{x,y}) \cap L^{\infty}([T^{-1},1], H^m_x L^2_y):  \sup_{t \in[T^{-1}, 1] }\| \nu(t) \|_{ H^m_x L^2_y}  \le R \right\}
	\]
	and the metric is given by
	\[
	d(\nu^{(1)},\nu^{(2)}) := \sup_{t \in [T^{-1}, 1]} \| \nu^{(1)}(t) - \nu^{(2)}(t) \|_{L^2_{x,y}}.
	\]
	We first show that $\Phi$ maps $Z_{T,R}$ to itself. Let $\nu\in Z_{T,R}$. Following the proof of \Cref{coro:NL}, we get
	\begin{align*}
		\| \Phi(\nu)(t) \|_{H^m_x L^2_y}
		&\leq \| \nu_1 \|_{H^m_x L^2_y} + \int_t^1 \tau^{-1}\|  V_\nu(\tau) \nu (\tau) \|_{H^m_x L^2_y} \, \dd \tau \\
		&\leq \frac{R}{2} + C T\int_t^1  \left(\| \nu (\tau) \|_{H^m_x L^2_y}^3 + \| \nu(\tau)\|_{H^m_x L^2_y}^{1+2/d} \right) \, \dd \tau \\
		&\leq \frac{R}{2} + C(T-1)(R^3 + R^{1+2/d})
	\end{align*}
	for $t \in [T^{-1}, 1]$.
	Choosing $T=T(R)>1$ sufficiently close to $1$, we have $\Phi(\nu) \in Z_{T,R}$, which proves
	\[
	\Phi(Z_{T,R}) \subset Z_{T,R}.
	\]
	We now prove that $\Phi$ is a contraction on $Z_{T,R}$. For two distinct $\nu^{(1)}$ and $\nu^{(2)}$ in $Z_{T,R}$, we write
	\begin{equation*}
		\Phi(\nu^{(1)})(t) - \Phi(\nu^{(2)})(t)
		= -\ii \int_1^t \ee^{\ii (t-\tau) \Delta/2} \tau^{-1} (V_{\nu^{(1)}} \nu^{(1)} - V_{\nu^{(2)}} \nu^{(2)} )(\tau)\, \dd \tau.
	\end{equation*}
	Taking the $L^2_{x,y}$ norm, we obtain for $t \in [T^{-1}, 1]$ that
	\begin{equation}\label{eqn:LWP-diff}
		\| \Phi(\nu^{(1)})(t) - \Phi(\nu^{(2)})(t) \|_{L^2_{x,y}}
		\le T\int_t^1 \| V_{\nu^{(1)}} \nu^{(1)}(\tau)- V_{\nu^{(2)}} \nu^{(2)}(\tau) \|_{L^2_{x,y}} \, \dd \tau.
	\end{equation}
	We recall $V=V_{\mathrm{H}} + V_{\mathrm{xc}}$ and estimate the corresponding nonlinear terms separately. For the Hartree term, we write
	\begin{align*}
		V_{\mathrm{H},\nu^{(1)}} \nu^{(1)} - V_{\mathrm{H},\nu^{(2)}} \nu^{(2)}
		&= V_{\mathrm{H},\nu^{(1)}}(\nu^{(1)} - \nu^{(2)}) + (V_{\mathrm{H},\nu^{(1)}}-V_{\mathrm{H},\nu^{(2)}}) \nu^{(2)}.
	\end{align*}
	Recall from \Cref{est:interpolation-riesz} that
	\[
	\| |\cdot|^{-1} \ast_x \rho \|_{L^{\infty}(\mathbb{R}^d)} \lesssim \| \rho \|_{L^1(\mathbb{R}^d)}^{(d-1)/d} \| \rho \|_{L^{\infty}(\mathbb{R}^d)}^{1/d}.
	\]
	Applying this for $\rho= \rho_{\nu^{(1)}}$ and using Sobolev embedding, it follows that
	\[
	\| V_{\mathrm{H},\nu^{(1)}} (\nu^{(1)} - \nu^{(2)}) \|_{L^2_{x,y}}
	\lesssim \| V_{\mathrm{H},\nu^{(1)}} \|_{L^{\infty}_x} \| \nu^{(1)} - \nu^{(2)} \|_{L^2_{x,y}}
	\lesssim \| \nu^{(1)} \|_{H^m_x L^2_y}^2 \| \nu^{(1)} -\nu^{(2)} \|_{L^2_{x,y}}.
	\]
	Next, we write
	\[
	(V_{\mathrm{H},\nu^{(1)}} - V_{\mathrm{H},\nu^{(2)}} ) \nu^{(2)}
	= (|\cdot|^{-1} \ast_x \rho_{\nu^{(1)}, \nu^{(2)}}) \nu^{(2)},
	\]
	where
	\[
	\rho_{\nu^{(1)}, \nu^{(2)}}(x) = \langle \nu^{(1)} - \nu^{(2)} , \nu^{(1)}\rangle_{L^2} (x)  + \langle \nu^{(2)}, \nu^{(1)} - \nu^{(2)} \rangle_{L^2}(x).
	\]
	Using H\"older's inequality, we get
	\begin{align*}
		\| \rho_{\nu^{(1)}, \nu^{(2)}} \|_{L^1} &\lesssim \left(\| \nu^{(1)} \|_{L^2_{x,y}} + \| \nu^{(2)} \|_{L^2_{x,y}} \right) \| \nu^{(1)} - \nu^{(2)} \|_{L^2_{x,y}},\\
		\| \rho_{\nu^{(1)}, \nu^{(2)}} \|_{L^2} &\lesssim \left(\| \nu^{(1)} \|_{L^{\infty}_x L^2_y} + \| \nu^{(2)} \|_{L^{\infty}_x L^2_y} \right) \| \nu^{(1)} - \nu^{(2)} \|_{L^2_{x,y}}.
	\end{align*}
	Using Sobolev embedding and the definition of $Z_{T,R}$, we may bound
	\[
	\| \rho_{\nu^{(1)}, \nu^{(2)}} \|_{L^1 \cap L^2} \lesssim  R \| \nu^{(1)} - \nu^{(2)} \|_{L^2_{x,y}}.
	\]
	Choose $1 < p < d/(d-1)$, $q \ge 2$, and $r \ge 2$ satisfying
	\[
	\frac{1}{q} = \frac{1}{p} - \frac{d-1}{d}, \qquad \frac{1}{2} = \frac{1}{q} + \frac{1}{r}.
	\]
	By Hardy--Littlewood--Sobolev and H\"older inequalities, we get
	\begin{align*}
		\| (|\cdot|^{-1} \ast_x \rho_{\nu^{(1)}, \nu^{(2)}} ) \nu^{(2)} \|_{L^2_{x,y}}
		&\lesssim \| |\cdot|^{-1} \ast_x \rho_{\nu^{(1)}, \nu^{(2)}} \|_{L^q_x} \| \nu^{(2)} \|_{L^r_x L^2_y} \\
		&\lesssim \| \rho_{\nu^{(1)}, \nu^{(2)}} \|_{L^p_x} \| \nu^{(2)} \|_{L^r_x L^2_y} \\
		&\lesssim R^2 \| \nu^{(1)} - \nu^{(2)} \|_{L^2_{x,y}} .
	\end{align*}
	Combining these estimates, we obtain that
	\begin{equation}\label{est:LWP-diff1}
		\| V_{\mathrm{H},\nu^{(1)}} \nu^{(1)} - V_{\mathrm{H},\nu^{(2)}} \nu^{(2)}\|_{L^2_{x,y}}
		\lesssim R^2 \| \nu^{(1)} - \nu^{(2)} \|_{L^2_{x,y}}.
	\end{equation}
	For $V_{\mathrm{xc},\nu^{(1)}} \nu^{(1)} - V_{\mathrm{xc},\nu^{(2)}} \nu^{(2)}$, we recall the map $F_{\mathrm{xc}}(\varphi) = \| \varphi\|_{L^2_y}^{2/d} \varphi$ satisfies $\| DF_{\mathrm{xc}} (\varphi) \|_{\mathrm{op}} \lesssim \| \varphi \|_{L^2}^{2/d}$. This yields
	\begin{equation}\label{est:LWP-diff2}
		\begin{aligned}
			\| V_{\mathrm{xc},\nu^{(1)}} \nu^{(1)} - V_{\mathrm{xc},\nu^{(2)}} \nu^{(2)} \|_{L^2_{x,y}}
			&\lesssim \left( \| \nu^{(1)} \|_{L^{\infty}_x L^2_y}^{2/d} + \| \nu^{(2)} \|_{L^{\infty}_x L^2_y}^{2/d} \right) \| \nu^{(1)} - \nu^{(2)} \|_{L^2_{x,y}}\\
			&\lesssim R^{2/d} \| \nu^{(1)} - \nu^{(2)} \|_{L^2_{x,y}}.
		\end{aligned}
	\end{equation}
	Applying \Cref{est:LWP-diff1} and \Cref{est:LWP-diff2}, we obtain from \Cref{eqn:LWP-diff} that
	\begin{align*}
		\| \Phi(\nu^{(1)})(t) - \Phi(\nu^{(2)})(t) \|_{L^2_{x,y}}
		&\lesssim T\int_t^1 \sum_{j\in\{\mathrm{H},\mathrm{xc}\}} \| V_{j,\nu^{(1)}} \nu^{(1)} (\tau) - V_{j,\nu^{(2)}} \nu^{(2)} (\tau) \|_{L^2_{x,y}} \, \dd \tau  \\
		&\lesssim (T-1) (R^2 + R^{2/d} ) \| \nu^{(1)} - \nu^{(2)} \|_{L^2_{x,y}}
	\end{align*}
	for $t \in [T^{-1} ,1]$. Choosing $T=T(R)$ closer to $1$, if necessary, we obtain that the map $\Phi$ is a contraction on $Z_{T,R}$. By the contraction mapping principle, there is a unique solution $$\nu \in C([T^{-1}, 1], L^2_{x,y}) \cap L^{\infty}([T^{-1}, 1], H^m_x L^2_y)$$ of \eqref{eqn:nu}. Since $\ee^{\ii t\Delta}$ is strongly continuous on $H^m$ and $V_{\nu} \nu(\tau) \in L^1([T^{-1}, 1], H^m_x L^2_y)$, Duhamel's formula implies that $\nu \in C([T^{-1}, 1], H^m)$. This completes the proof of \Cref{lem:LWP}.
\end{proof}

\section{Subcritical Kohn--Sham Equations}\label{sec:subcritical}

Consider the subcritical time-dependent Kohn--Sham equation
\begin{equation}\label{eqn:KS-subcrit-gamma}
	\ii \partial_t \gamma = \Big[ -\frac{1}{2} \Delta + \lambda (|x|^{-\alpha} \ast \rho_{\gamma}) + \mu \rho_{\gamma}^{\beta}, \, \gamma \Big],
\end{equation}
with $1 < \alpha < d $ and $\beta >1/d$, where $d \le 3$. In \cite{Pusateri2021}, almost all subcritical regimes were considered, namely, $\beta > 1/ \min \{ d, 2\}$. This excludes the case when $\beta \in (1/3, 1/2]$ in $d=3$. In this section, we address these cases for completeness.

\begin{theorem}\label{thm:subcritical}
	Consider \eqref{eqn:KS-subcrit-gamma} with $d=3$, $\alpha \in (1,3/2)$ and $\beta \in(1/3,1/2]$. Let $m \in (3/2, \min\{1+2\beta, \, 3-\alpha\})$ and assume that the initial density matrix $\gamma_1$ satisfies
	\[
	\| \langle \mathrm{J}(1) \rangle^m \gamma_1 \langle \mathrm{J}(1) \rangle^m \|_{\mathfrak{S}^1} + \| \rho_{\gamma_1} \|_{L^{\infty}_x} \le \epsilon_0.
	\]
	For sufficiently small $\epsilon_0>0$, there exists a unique global-in-time solution $\gamma \in C ([1,\infty), \mathfrak{S}^1)$
	to \Cref{eqn:KS-subcrit-gamma} satisfying $|\mathrm{J}|^m \gamma |\mathrm{J}|^m \in C ([1,\infty), \mathfrak{S}^1)$ and
	\[
	\| \rho_{\gamma}(t) \|_{L^p_x} \lesssim \epsilon_0  t^{-3(1-1/p)}
	\]
	for $p \in [1,\infty]$ and $t \ge 1$.
	Moreover, there exists a unique $\gamma_{\infty} \in \mathfrak{S}^1$ such that, for some $\delta' >0$,
	\[
	\| \gamma(t)- \ee^{\ii t\Delta/2}\gamma_{\infty} \ee^{-\ii t\Delta/2} \|_{\mathfrak{S}^1}
	\lesssim \epsilon_0  t ^{-\delta'}
	\]
	holds for all $t \ge 1$.
\end{theorem}

\begin{remark}
	The restricted range of $1 < \alpha < 3/2$ comes from the application of the Hardy--Littlewood--Sobolev inequality \Cref{eqn:HLS} in the proof. The whole range of $1<\alpha<3$ might be covered by using Strichartz estimates.
\end{remark}

As before in \Cref{sec:main-result}, we use the formulation in $\kappa$ given by
\begin{equation}\label{eqn:KS-subcrit-u}
	\ii \partial_t \kappa(t,x,y)  =-\frac12\Delta_x \kappa(t,x,y)+ \lambda (|x|^{-\alpha}\ast_x \rho_{\kappa\kappa^{\ast}}(t))(x) \kappa(t,x,y) +\mu \rho_{\kappa\kappa^{\ast}}^{\beta}(t,x) \kappa(t,x,y),
\end{equation}
with $\rho_{\kappa\kappa^{\ast}}(x) = \| \kappa(x,\cdot) \|_{L^2_y}^2$, and it is sufficient to show the following theorem for this formulation.
\begin{theorem}
	Consider \Cref{eqn:KS-subcrit-u} with $d=3$, $\alpha \in (1,3/2)$ and $\beta \in(1/3,1/2]$. Let $m \in (3/2, \min\{1+2\beta, \, 3-\alpha\})$ and assume that the initial data $\kappa_1$ satisfies
	\[
	\| \kappa_1 \|_{L^{\infty}_x L^2_y} + \| \langle \mathrm{J}_x(1) \rangle^m \kappa_1 \|_{L^2_{x,y}} \le \epsilon_0.
	\]
	For sufficiently small $\epsilon_0>0$, there exists a global-in-time solution $\kappa \in C([1,\infty), \mathfrak{S}^2)$ to \Cref{eqn:KS-subcrit-u} satisfying $|\mathrm{J}_x|^m \kappa \in C([1,\infty), \mathfrak{S}^2)$ and
	\[
	\| \kappa(t) \|_{L^p_x L^2_y} \lesssim \epsilon_0 t^{-3(1/2-1/p)}
	\]
	for $p \in [2,\infty]$ and $t \ge 1$.
	Moreover, there exists a unique $\kappa_\infty \in L^2_{x,y}$ such that, for some $\delta>0$,
	\[
	\| \kappa(t) - \ee^{\ii t\Delta_x/2} \kappa_{\infty}  \|_{L^2_{x,y}}
	\lesssim \epsilon_0  t^{-\delta}
	\]
	holds for all $t \ge 1$.
\end{theorem}

\begin{proof}
	We recall the pseudo-conformal transformation introduced in \Cref{subsec:pseudo-conformal}:
	\[
	\omega(t,x,y) = (\ii t )^{d/2} \ee^{-\ii t|x|^2/2} \kappa(t,tx,y), \qquad \nu(t,x,y) = \overline{\omega(1/t,x,y)}, \qquad t \in (0,1].
	\]
	Then \Cref{eqn:KS-subcrit-u} becomes
	\[
	\ii\partial_t \nu(t)
	= - \frac{1}{2} \Delta_x \nu(t) + \left( \lambda t^{-2 + \alpha} (|x|^{-\alpha} \ast_x \rho_{\nu\nu^{\ast}}) (t) + \mu t^{-2 + 3 \beta} \rho_{\nu\nu^{\ast}}^{\beta}(t) \right) \nu(t), \qquad t \in (0,1].
	\]
	Using \Cref{eqn:pseudo-conformal-prop}, we have
	\[
	\| \nu_1 \|_{H^m_x L^2_y} \le \epsilon_0.
	\]
	We omit the proof of local well-posedness for \Cref{eqn:KS-subcrit-u}, which can be done similarly as in \Cref{sec:LWP}. For the local-in-time solution $\nu(t)$ in $C([T^{-1},1], H^m_x L^2_y)$, we make the following bootstrap assumption.
	\[
	\sup_{t \in [T^{-1} , 1]}
	\| \nu(t) \|_{H^m_x L^2_y}
	\le \epsilon_1.
	\]
	To close the bootstrap argument, we prove that
	\begin{equation}\label{eqn:subcrit-apriori}
		\sup_{t \in [T^{-1}, 1]} \| \nu(t) \|_{H^m_x L^2_y} \lesssim \epsilon_0 + \epsilon_1^{1+2\beta}.
	\end{equation}
	Observe that $\nu(t)$ satisfies
	\[
	\nu(t) = \ee^{\ii(t-1)\Delta_x/2} \nu_1 -\ii \int_1^t \ee^{\ii(t-\tau)\Delta_x/2} \left( \lambda \tau^{-2 + \alpha} (|x|^{-\alpha} \ast_x \rho_{\nu\nu^{\ast}} (\tau)) + \mu \tau^{-2 +3 \beta} \rho_{\nu\nu^{\ast}}^{\beta}(\tau,x) \right) \nu(\tau) \, \dd\tau
	\]
	for $T^{-1} \le t \le 1$. Following the proof in \Cref{lem:potential1v}, we split the integral at $R = \| \rho_{\nu\nu^{\ast}} (\tau)\|_{L^1}^{1/3} \| \rho_{\nu\nu^{\ast}} (\tau)\|_{L^{\infty}}^{-1/3}$, which yields
	\[
	\| |x|^{-\alpha} \ast \rho_{\nu\nu^{\ast}} (\tau) \|_{L^{\infty}} \lesssim \| \rho_{\nu\nu^{\ast}} (\tau) \|_{L^1}^{(3-\alpha)/3} \| \rho_{\nu\nu^{\ast}} (\tau) \|_{L^{\infty}}^{\alpha/3} \lesssim \epsilon_1^2.
	\]
	For the $\dot{H}^m_x L^2_y$ bound, we note that $\partial_x^m(|x|^{-\alpha} \ast \rho_{\nu\nu^{\ast}} ) = \partial_x^m (-\Delta_x)^{-(3-\alpha)/2} \rho_{\nu\nu^{\ast}}$ with $ m - 3 +\alpha \in(-1/2, 0)$. Applying Hardy--Littlewood--Sobolev inequality, we have
	\begin{equation}\label{eqn:HLS}
		\| |x|^{-\alpha} \ast \rho_{\nu\nu^{\ast}} (\tau) \|_{\dot{H}^m} \lesssim \| \rho_{\nu\nu^{\ast}} (\tau) \|_{L^{6/(9-2\alpha-2m)}} \lesssim \epsilon_1^2.
	\end{equation}
	It follows from the fractional Leibniz rule that
	\begin{align*}
		\left\| \left(|x|^{-\alpha} \ast_x \rho_{\nu\nu^{\ast}} \right)\nu(\tau) \right\|_{H^m_x L^2_y}
		&\lesssim \left\| |x|^{-\alpha} \ast_x \rho_{\nu\nu^{\ast}}(\tau) \right\|_{L^{\infty}} \| \nu (\tau) \|_{H^m_x L^2_y}
		+\left\| |x|^{-\alpha} \ast_x \rho_{\nu\nu^{\ast}}(\tau) \right\|_{\dot{H}^m} \| \nu (\tau) \|_{L^{\infty}_x L^2_y}
		\\
		&\lesssim \epsilon_1^3.
	\end{align*}
	
	Using arguments similar to those in the proof of \Cref{lem:powertype}, one can prove an analogous bound for $F(\varphi)= \| \varphi\|_{L^2}^{2\beta} \varphi$ with $\beta \in (1/3, 1/2]$ in three dimensions. Hence we bound
	\begin{align*}
		\left\| \rho_{\nu\nu^{\ast}}^{\beta} \nu(\tau) \right\|_{H^m_x L^2_y} \lesssim \| \nu(\tau) \|_{L^{\infty}_x L^2_y}^{2\beta} \|\nu(\tau)\|_{H^m_x L^2_y} \lesssim \epsilon_1^{1+2\beta}.
	\end{align*}
	Combining these estimates, we obtain that
	\[
	\| \nu(t) \|_{H^m_x L^2_y}
	\lesssim \epsilon_0 +   \int_t^1 \left(\epsilon_1^3 \tau^{-2+\alpha} + \epsilon_1^{1+2\beta} \tau^{-2+3\beta} \right) \dd\tau \lesssim \epsilon_0 + \epsilon_1^{1+2\beta},
	\]
	which completes the proof of \Cref{eqn:subcrit-apriori}.
	
	To establish scattering, we consider the profile $\eta(t) =  \ee^{-\ii t\Delta_x/2} \nu(t)$. Then we can write
	\[
	\ii \partial_t \eta(t)
	= \ee^{-\ii t\Delta_x/2} \left( \lambda t^{-2 + \alpha} (|x|^{-\alpha} \ast_x \rho_{\nu\nu^{\ast}} (t)) + \mu t^{-2 +3 \beta} \rho_{\nu\nu^{\ast}}^{\beta}(t) \right) \nu(t)
	\]
	so that Sobolev embedding and Duhamel principle yield
	\begin{equation}\label{est:subcrit-t1t2}
		\begin{aligned}
			\| \eta(t_1) - \eta(t_2) \|_{(L^2_x \cap L^{\infty}_x)L^2_y} &\lesssim \int_{t_2}^{t_1} \left( \tau^{-2+\alpha} \| (|x|^{-\alpha} \ast \rho_{\nu\nu^{\ast}}) \nu(\tau) \|_{H^m_x L^2_y} + \tau^{-2 + 3\beta} \| \rho_{\nu\nu^{\ast}}^{\beta} \nu (\tau) \|_{H^m_x L^2_y} \right) \, \dd \tau \\
			&\lesssim \int_{t_2}^{t_1} \left( \epsilon_1^3 \tau^{-2+\alpha} + \epsilon_1^{1+2\beta} \tau^{-2+3\beta} \right) \, \dd \tau \\
			&\lesssim \epsilon_1^3 \left(t_1^{-1+\alpha} - t_2^{-1+\alpha}\right) + \epsilon_1^{1+2\beta} \left(t_1^{-1+3\beta} - t_2^{-1+3\beta}\right)
		\end{aligned}
	\end{equation}
	for $0< t_2 \le t_1 \le 1$. Since $\alpha >1$ and $\beta>1/3$, the limit of $\eta(t)$ exists as $t \rightarrow 0^{+}$ in $(L^2_x \cap L^{\infty}_x)L^2_y$. We denote it by $\eta(0)$. Inverting the conjugations and pseudo-conformal transformation, we get
	\begin{align*}
		\kappa(t) = \ee^{\ii t\Delta_x/2} \mathcal{F}^{-1} \overline{\eta(1/t)}.
	\end{align*}
	Since $\ee^{\ii t \Delta_x/2}$ is unitary on $L^2_x$, we obtain from \Cref{est:subcrit-t1t2} that
	\[
	\| \kappa(t) - \ee^{\ii t\Delta_x/2} \mathcal{F}^{-1} \overline{\eta(0)}\|_{L^2_{x,y}}
	\lesssim \| \eta(1/t) - \eta(0) \|_{L^2_{x,y}}
	\lesssim \epsilon_1^{1+2\beta} (t^{1-\alpha} + t^{1-3\beta}),
	\]
	which in particular proves scattering by setting $\kappa_{\infty}:= \mathcal{F}^{-1} \overline{\eta(0)} \in L^2_{x,y}$.
\end{proof}

\section*{Acknowledgment}
Masaki Kawamoto is partially supported by JSPS KAKENHI Grant Number JP24K06796 and JP26K00612.
Jinyeop Lee is partially supported by Global - Learning \& Academic research institution for Master's$\cdot$PhD students, and Postdocs(G-LAMP) Program of the National Research Foundation of Korea(NRF) grant funded by the Ministry of Education (No. RS-2025-25442355), and by a grant from Kyung Hee University in 2026 (KHU-20262263).
Chanjin You is partially supported by the NSF under grant DMS-2349981.

\bibliographystyle{abbrv}

\end{document}